\documentclass[]{amsart}
\usepackage{amsmath}
\usepackage{amsfonts}
\usepackage{amssymb}
\usepackage{txfonts}
\usepackage[T1]{fontenc}
\usepackage{hyperref}
\usepackage{epsfig}
\newtheorem{theorem}{Theorem}
\newtheorem{corollary}{Corollary}
\newtheorem{definition}{Definition}
\newtheorem{example}{Example}
\newtheorem{lemma}{Lemma}
\newtheorem{proposition}{Proposition}
\newtheorem{remark}{Remark}
\newcommand{\eps}{\varepsilon}
\DeclareMathOperator{\dom}{dom}

\DeclareMathOperator{\coker}{coker}
\DeclareMathOperator{\ev}{ev}
\DeclareMathOperator{\im}{Im}

\DeclareMathOperator{\fix}{Fix}

\DeclareMathOperator{\co}{co}

\DeclareMathOperator{\bd}{bd}
\DeclareMathOperator{\Int}{int}
\DeclareMathOperator{\supp}{supp}
\DeclareMathOperator{\sgn}{sgn}
\DeclareMathOperator{\Liminf}{Lim\;inf}
\newcommand{\w}{\tilde}
\newcommand{\map}{\multimap}
\newcommand{\<}{\leqslant}
\newcommand{\n}{{n\geqslant 1}}
\newcommand{\K}{{k\geqslant 1}}
\DeclareMathOperator{\F}{F}
\newcommand{\R}[1]{\mathbb{R}^{#1}}

\newcommand{\res}[2]{#1\:\rule[-1.5mm]{0.45pt}{4mm}\,\rule[-1mm]{0mm}{4mm}_{#2}}

\begin{document}
\title[On nonlocal Cauchy problems for constrained differential inclusions]{On nonlocal Cauchy problems for constrained differential inclusions in Euclidean space}
\author{Rados\l aw Pietkun}
\address{Kwitn{\k a}cej Jab\l oni 15, 87-100 Toru\'n, Poland}
\email{rpietkun@pm.me}

\begin{abstract}
We investigate the existence of solutions of constrained nonlinear differential inclusions with nonlocal boundary conditions. Our viability theorems are based on the assumption that the right-hand side of differential inclusion is defined on the domain possessing a certain type of geometric regularity, expressed in terms of locally Lipschitz functional constraints. For solvability of the Floquet boundary value problems associated with differential inclusions we engage the bound set technique. It relies on the usage of not necessarily differentiable bounding functions.
\end{abstract}
\maketitle

\section{Introduction}
\par Given a compact real interval $I=[0,T]$ and a multivalued map $F\colon I\times K\map\R{N}$ defined on a closed subset $K$ of the $N$-dimensional Euclidean space $\R{N}$, we look for solutions of the following constrained nonlocal Cauchy problem
\begin{equation}\label{inclusion}
\begin{gathered}
x'\in F(t,x),\;\text{a.e. on } I,x\in K\\
x(0)=g(x),
\end{gathered}
\end{equation}
where $g\colon C(I,\R{N})\to\R{N}$ is a continuous boundary operator, specified later. The main results of this paper were obtained by means of a topological fixed point theory and coincidence degree arguments.
\par In the case of a standard initial condition, i.e. when $g$ is a constant function, the problem \eqref{inclusion} is at the core of interest of the so-called viability theory, which has been already devoted many monographs (see \cite{aubin2,aubin}). When a constrained (multivalued) differential equation is enriched with nonlocal boundary condition, then the problem becomes much more sophisticated. The very idea of a systematic study of nonlocal Cauchy problems, with the right-hand side not subjected to any additional state constraints, was initiated by Byszewski and Lakshmikantham. They proved, about 1991, the existence and uniqueness of mild solutions for nonlocal semilinear differential equations in Banach spaces. The consideration for nonlocal initial condition is stimulated by the observation that this type of conditions is more realistic than usual ones in treating physical problems. Some typical examples of nonlocal initial conditions, to which we refer in the course of this work, are: $g(x)=x(T)$ (periodicity condition); $g(x)=-x(T)$ (antiperiodic condition); $g(x)=Cx(T)$ with $C\in GL(N,\R{})$ (Floquet condition), $g(x)=\sum_{i=1}^n\alpha_ix(t_i)$, where $\sum_{i=1}^n|\alpha_i|\<1$ and $0<t_1<\ldots<t_n\<T$ (multi-point discrete mean condition); $g(x)=\frac{1}{T}\int_0^Th(x(t))\,dt$, with $h\colon\R{N}\to\R{N}$ such that $|h(x)|\<|x|$ (mean value condition).\par Section 2. contains basic definitions and preliminary results. Subsequently, in Section 3. we show (in Theorem \ref{th9}.) that there is a solution of nonlocal Cauchy problem \eqref{inclusion}, which is a viable trajectory in the set $K=f^{-1}((-\infty,0])$ of functional constraints, possessing the geometry of the so-called strictly regular set (see Definition \ref{K}). In this distinguished class of state constraints there are both nonconvex sets and sets with empty interiors. Customary used tangential conditions formulated in terms of the intersection with the Bouligand contingent cone are not sufficient for our purposes. Therefore, relying on ideas contained in the work of Bader and Kryszewski (\cite{bader}), we use a new and somewhat restrictive tangential condition expressed in terms of the polar cone to the generalized Clarke gradient of a locally Lipschitz map $f$ representing the set of state constraints. The proof of the aforementioned existence result refers to the knowledge about the topological structure of the solution set of differential inclusions with constraints. \par Further in the paper we focus on the issue of the existence of antiperiodic solutions of constrained differential inclusions, aiming to weaken the assumptions concerning geometrical properties of the constraint set (i.e. that it is contractible and strictly regular). The technique that we use to prove the existence of antiperiodic solutions of differential inclusion, whose set of functional constraints is represented by a mapping $f$ having no critical points outside the $0$-sublevel set of $f$ (Theorem \ref{ant}.), consists in indicating approximate solutions which satisfy relaxed constraints, i.e. they are viable in appropriately ``thickened'' $\eps$-sublevel set (for $\eps>0$). To this end, we show previously Theorem \ref{th1}. and \ref{borsuk1}., which ensure the existence of antiperiodic solutions on the sets of epi-Lipschitz type, provided the suitable tangency conditions are met. Applied by us tangential conditions assume a particularly simplified form 
in the case where a ball centered at origin plays the role of the viability domain of the right-hand side of \eqref{inclusion} (see Theorem \ref{th2}.).\par The bound sets approach which we use in the following part of Sections 3. of our paper originates from Gaines and Mawhin (\cite{maw2}), who applied it for obtaining the existence of periodic oscillations of first-order as well as second-order ordinary differential equations. The notion of a bounding function, closely related to the concept of a guiding function, was systematically used for the study of multivalued first-order Floquet problems, as well as other boundary value problems, by Andres et al. (see \cite{andres1,andres2,andres3}). Our approach to the Floquet boundary value problem also uses the concept of bounding function (Theorem \ref{floqth}.), but the viability result obtained by us is maintained more in the spirit of \cite{mawhin}. Given by us generalization of the definition of an autonomous bound set for multivalued Floquet problem is correlated with the sign of the Clarke directional derivative (in contrast to \cite{andres1}, where the directional derivative in the sense of Penot is utilized). Our work concludes with the observation (Theorem \ref{floqth2}.) that the presence of natural tangential conditions expressed in term of vectors normal to an open set, possessing a completely arbitrary geometry, implies the existence of bounding functions in the sense that ensures the solvability of the Floquet boundary value problem.\mbox{}\vspace{-0.2cm}
\section{Notations and auxiliary results}
Let $X$ and $Y$ be a Banach space. An open (resp. closed) ball with center $x\in X$ and radius $r>0$ is denoted by $B(x,r)$ ($D(x,r)$). If $A\subset X$, then $\overline{A}$ (resp. $\Int A$, $\bd A$, $\co A$) denotes the closure (the interior, the boundary, the convex hull) of $A$. $|\cdot|$ is the standard norm in $\R{N}$, i.e. $|x|=\sqrt{\langle x,x\rangle}$, where $\langle\cdot,\cdot\rangle$ is the inner product. 
\par For $I\subset\R{}$, $(C(I,\R{N}),||\cdot||)$ is the Banach space of continuous maps $I\to \R{N}$ equipped with the maximum norm and $AC(I,\R{N})$ is the subspace of absolutely continuous functions. By $(L^1(I,\R{N}),||\cdot||_1)$ we mean the Banach space of all Lebesgue integrable maps.
\par A multivalued map $F\colon X\map Y$ assigns to any $x \in X$ a nonempty subset $F(x)\subset Y$. The set of all fixed points of the multivalued (or univalent) map $F$ is denoted by $\fix(F)$. A set-valued map $F\colon X\map Y$ is called upper (resp. lower) semicontinuous (usc or lsc in short) if $\{x\in X\colon F(x)\subset U\}$ is open (closed) in $X$ whenever $U$ is open (closed) in $Y$. If the image $F(X)$ is relatively compact in $Y$, then we say that $F$ is a compact multivalued map. We say that $F\colon I\map\R{N}$ is measurable, if $\{t\in I\colon F(t)\subset A\}$ belongs to the Lebesgue $\sigma$-field of $I$ for every closed $A\subset\R{N}$. We shall call $F\colon I\times X\map Y$ a lower Carath\'eodory type multivalued mapping if it satisfies:
\begin{itemize}
\item[(i)] the multimap $F(t,\cdot)$ is lsc for each fixed $t\in I$, and 
\item[(ii)] the multimap $F(\cdot,\cdot)$ is product measurable on $I\times X$, with $X$ equipped with its Borel structure.
\end{itemize}
\par A set-valued map $F\colon X\map Y$ is admissible (comp. \cite[Def.40.1]{gorn}) if there is a metric space $Z$ and two continuous functions $p\colon X\to Z$, $q\colon Z\to Y$ from which $p$ is a Vietoris map such that $F(x)=q(p^{-1}(x))$ for every $x\in X$. It turns out that every acyclic multivalued map, i.e. an usc multimap with compact acyclic values, is admissible. In particular, every usc multivalued map with compact convex values is admissible.
\par Let ${\mathcal M}$ be the set of triples $(Id-F,\Omega,y)$ such that $\Omega\subset X$ is open bounded, $Id$ is the identity, $F\colon\overline{\Omega}\map X$ is a compact usc multimap with closed convex values, and $y\not\in(Id-F)(\partial\Omega)$. Then it is possible to define, using approximation methods for multivalued maps, a unique topological degree function $\deg\colon{\mathcal M}\to\mathbb{Z}$ (see \cite{deimling,gorn} for details). This degree inherits directly all the basic properties of the Leray-Schauder degree.
\par Let $H_*$ be the singular homology functor with rational coefficients from the category of topological spaces and continuous maps to the category of graded vector spaces and linear maps of degree zero. If $X$ is a topological space, then $H_*(X)=\{H_q(X)\}_{q\geqslant 0}$ is a graded vector space, $H_q(X)$ being the $q$-dimensional singular homology group of $X$. For a continuous mapping $f\colon X\to Y$, $f_*=\{f_{q*}\}_{q\geqslant 0}$ is the induced linear map, where $f_{q*}\colon H_q(X)\to H_q(Y)$. It is a well-known fact that if $X$ is a compact absolute neighbourhood retract (ANR), then $X$ is of finite type, i.e. $\dim H_q(X)<\infty$ for all $q$ and $H_q(X)=0$ for almost all $q$. Thus, the Euler-Poincar\'e characteristic $\chi(X):=\sum_{q\geqslant 0}(-1)^q\dim H_q(X)$ is a well-defined integer.
\par Let $f\colon\dom(f)\to\R{}$ be a locally Lipschitz function, defined on open subset $\dom(f)$ of the Euclidean space $\R{N}$. The upper (resp. lower) directional derivative of $f$ at $x\in\dom(f)$ in the direction $v\in\R{N}$ in the sense of Clarke is defined by
\[f^\circ(x;v):=\underset{y\,\to x,\,h\,\to\, 0^+}{\limsup}\,\frac{f(y+hv)-f(y)}{h}\;\;\;\left(f_\circ(x;v):=\underset{y\,\to x,\,h\,\to\, 0^+}{\liminf}\,\frac{f(y+hv)-g(y)}{h}\right).\] Remind that $f$ is said to be (directionally) regular at $x$ provided, for every $v$, the usual one-sided directional derivative $f'(x;v)$ exists and $f'(x;v)=f^\circ(x;v)$. The generalized gradient of $f$ at $x$ is defined as a set
\[\partial f(x):=\left\{p\in\R{N}\colon\langle p,v\rangle\<f^\circ(x;v)\;\forall\,v\in\R{N}\right\}.\] It is well known that the mapping $\dom(f)\times\R{N}\ni(x,v)\mapsto f^\circ(x;v)$ is upper semicontinuous (as a univalent map), while the function $\R{N}\ni v\mapsto f^\circ(x;v)$ is Lipschitz continuous, subadditive and positively homogeneous. In turn, $\dom(f)\ni x\mapsto\partial f(x)\subset\R{N}$ is a compact convex valued upper semicontinuous set-valued map. The tangency conditions formulated in the course of Section 3. utilize the concept of the negative polar cone to $\partial f(x)$ \[\partial f(x)^\circ=\left\{v\in\R{N}\colon\langle p,v\rangle\<0\;\forall\,p\in\partial f(x)\right\}=\left\{v\in\R{N}\colon f^\circ(x;v)\<0\right\}.\] \par\noindent Let $f$ be as above.
\begin{definition}
We say that a closed subset $K\subset\R{N}$ is represented by the function $f$ or $f$ is a representing function of $K$ if $K$ coincides with the $0$-sublevel set of $f$, i.e. \[K=\{x\in\dom(f)\colon f(x)\<0\}.\]
\end{definition}
The geometrical regularity of sets forming a viability domain of the right-hand side of inclusion \eqref{inclusion}, to which we refer from the outset, characterizes the following:
\begin{definition}\label{K}
The closed set $K$, represented by a locally Lipschitz function $f\colon\dom(f)\to\R{}$, is said to be 
\begin{itemize}
\item[(i)] strongly regular, if $\;0\not\in\partial f(x)$ for every $x\in\bd K$, 
\item[(ii)] strictly regular, if \[\underset{\underset{{\scriptstyle y\notin K}}{y\to x}}{\liminf}\inf_{z\in\partial f(y)}|z\,|>0\] for every $x\in\bd K$,
\item[(iii)] regular, if there exists an open neighbourhood $U$ such that $K\subset U\subset\dom(f)$ and $0\not\in\partial f(x)$ for $x\in U\setminus K$.
\end{itemize}
\end{definition}
\begin{corollary}
It is clear that a strongly regular set is strictly regular. Furthermore, any strictly regular set is regular.
\end{corollary}
\par It should be emphasized that the type of regularity of the set $K$ significantly depends on the function which represents this set. The ambiguity of this representation causes that the set $K$ can be (strongly, strictly) regular with respect to one function without being it simultaneously with regard to other. In order to present the concepts contained in Definition \ref{K}. we provide (referring to \cite{bader}) the following:
\begin{example}\label{ex}\mbox{}\par
\begin{itemize}
\item[$(1)$] Any convex closed subset $K\subset\R{N}$, represented by a distance function $d_k$, is {strict\-ly} regular.
\item[$(2)$] Generally speaking, any proximate ratract $K$ of $\R{N}$ is strictly regular, being represented by $d_K$. The set $K$ is a proximate retract if it possesses such a neighbourhood $U$, for which the metric projection $P_K\colon U\map K$ onto set $K$ is a univalent $($and automatically continuous$)$ mapping. In particular closed convex sets, being Chebyshev, belong to the class of proximate retracts.
\item[$(3)$]  Every epi-Lipschitz subset $K$ of $\R{N}$ is strongly regular, as a set represented by the function $\Delta_K=d_K-d_{K^c}$ $($see \cite{cornet}$)$. Let us recall that $K$ is said to be epi-Lipschitz, if $\Int C_K(x)\neq\emptyset$ for each $x\in\bd K$. On the other hand, if $0\not\in\partial f(x)$, then there is a vector $v\in\R{N}$ such that $f^\circ(x;v)<0$. Therefore, a strongly regular set $K$ $($represented by a function $f)$ has the property that $\Int\partial f(x)^\circ\neq\varnothing$ for $x\in\bd K$. It is also known that $\partial f(x)^\circ\subset C_K(x)$, as far as $0\not\in\partial f(x)$ $($\cite[Th.2.4.7]{clarke}$)$. Hence the conclusion that $\Int C_K(x)\neq\varnothing$ for $x$ from the boundary of strongly regular $K$. In summary, strongly regular sets are also epi-Lipschitz.
\item[$(4)$] One may indicate such a mapping $f$, in respect of which an orientable closed manifold of $C^1$-class is strictly regular.
\item[$(5)$] The set $K=\{(x,y)\in\R{2}\colon (x-1)^2+y^2=1\vee (x+1)^2+y^2=1\}$, represented by its distance function $d_K$, is admittedly not strictly regular but it is a regular set.
\end{itemize}
\end{example}
Other fundamental concepts taken from the viability theory are the Bouligand contingent cone to $K\subset\R{N}$ at $x\in K$
\[T_K(x)=\left\{v\in\R{N}\colon\liminf_{h\,\to\, 0^+}\frac{d_K(x+hv)}{h}=0\right\}\] and the Clarke tangent cone to $K$ at $x$, i.e.
\[C_K(x)=\left\{v\in\R{N}\colon\underset{y\,\xrightarrow[]{K}x,\,h\,\to\, 0^+}{\lim}\frac{d_K(y+hv)}{h}=0\right\},\] where $d_K=d(\cdot,K)$ is the distance function. The cones in question are linked by the following relationship $\partial d_K(x)^\circ=C_K(x)\subset T_K(x)$. If the set $K$ is tangentially regular at $x$ (for instance, when it has a geometry of a proximate retract), then the Bouligand cone is closed convex and $C_K(x)=T_K(x)$. Recall also that so-called regular normals forming the Bouligand normal cone $N_K^b(x)$ are derived from polar to the contingent cone, i.e. $N_K^b(x):=T_K(x)^\circ$. \par The extension formula, for continuous (set-valued) maps defined on a closed subset of a metric space, applied several times throughout the next section rests on the following notion (comp. \cite[Lemma 3.1]{bessaga}):
\begin{definition}
If $K$ is a closed subset of a metric space $(X,d)$, then an indexed family $\{(U_s,a_s)\}_{s\in S}$ such that, for all $s\in S$,
\begin{itemize}
\item[(i)] $U_s\subset X\setminus K$, $a_s\in\bd K$,
\item[(ii)] if $x\in U_s$, then $d(x,a_s)\<2d_K(x)$,
\item[(iii)] $\{U_s\}_{s\in S}$ is a locally finite open covering of $X\setminus K$
\end{itemize}
is called a Dugundji system for the complement $X\setminus K$.
\end{definition}
\par For reader's convenience we recall some significant results that we shall need in the course of our considerations. The first is the following coincidence point theorem.
\begin{theorem}\label{conth}\cite[Lemma 13.1.]{deimling} Let $X$ and $Z$ be real Banach spaces, $L\colon\dom L\subset X\to Z$ be Fredholm operator of index zero and with closed graph, $\Omega\subset X$ be open bounded and $N\colon\overline{\Omega}\map Z$ be such that  $QN$ and $K_{P,Q}N$ are compact usc multimaps with compact convex values. Assume also that
\begin{itemize}
\item[(a)] $Lx\not\in\lambdaup N(x)$ for all $\lambdaup\in(0,1)$ and $x\in\dom L\cap\partial\Omega$,
\item[(b)] $0\not\in QN(x)$ on $\ker L\cap\partial\Omega$ and $\deg(\Phi QN,\ker L\cap\Omega,0)\neq 0$.
\end{itemize}   
Then $Lx\in N(x)$ has a solution in $\overline{\Omega}$.
\end{theorem}
\par The succeeding property is commonly known as the convergence theorem for upper hemicontinuous maps with convex values.
\begin{theorem}\label{convergence}
Let $F\colon X\map E$ be an upper hemicontinuous map from a metric space $X$ to the closed convex subsets of a Banach space $E$. If $I$ is a finite interval of $\,\R{}$ and sequences $(x_n\colon I\to X)_\n$ and $(y_n\colon I\to E)_\n$ satisfy the following conditions
\begin{itemize}
\item[(i)] $(x_n)_\n$ converges a.e. to a function $x\colon I\to X$,
\item[(ii)] $(y_n)_\n$ converges weakly in the space $L^1(I,E)$ to a function $y\colon I\to E$,
\item[(iii)] $y_n(t)\in\overline{\co}B(F(B(x_n(t),\eps_n)),\eps_n)$ for a.a. $t\in I$, where $\eps_n\to 0^+$ as $n\to\infty$,
\end{itemize}
then $y(t)\in F(x(t))$ for a.a. $t\in I$.
\end{theorem}
\par Another example of application of the coincidence degree theory is the following continuation theorem taken from \cite{gaines}. Throughout the rest of this paper $i\colon\R{N}\hookrightarrow C(I,\R{N})$ will denote a fixed embedding, given by $i(x_0)(t)=x_0$.
\begin{theorem}\label{maw}
Let $h\colon I\times\R{N}\to\R{N}$, $\gamma\colon C(I,\R{N})\to\R{N}$ be continuous mappings, with $\gamma$ taking bounded sets into bounded sets. Assume that the following conditions hold:
\begin{itemize}
\item[(i)] There exists an open bounded set $\Omega\subset C(I,\R{N})$ such that, for each $\lambdaup\in(0,1)$ and each possible solution $x$ to BVP
\[\begin{gathered}
\dot{x}(t)=\lambdaup\,h(t,x(t)),\;t\in I,\\
\gamma(x)=0,
\end{gathered}\]
one has $x\not\in\bd\Omega$.
\item[(ii)] Each possible solution $x_0\in\R{N}$ of the equation $(\gamma\circ i)(x_0)=0$ is such that $i(x_0)\not\in\bd\Omega$.
\item[(iii)] The Brouwer degree $\deg(\gamma\circ i,\w{\Omega},0)\neq 0$, where $\w{\Omega}:=\{x_0\in\R{N}\colon i(x_0)\in\Omega\}$.
\end{itemize}
Then the following boundary value problem
\[\begin{gathered}
\dot{x}(t)=h(t,x(t)),\;t\in I,\\
\gamma(x)=0,
\end{gathered}\]
has at least one solution $x\in\overline{\Omega}$.
\end{theorem}
The subsequent result is a straightforward consequence of a Lefschetz-type fixed point theorem \cite[Th.41.7]{gorn}.
\begin{theorem}\label{homo}
Let $X$ be an ANR and let $\psi\colon [0,1]\times X\map X$ be an admissible map such that $\psi(0,x)=\{x\}$ on $X$. If X is compact and $\chi(X)\neq 0$, then $\psi(1,\cdot)$ has a fixed point.
\end{theorem}
In what follows we shall permanently refer to a certain initial set of assumptions regarding the set-valued map $F\colon I\times K\map\R{N}$, which we itemize below:
\begin{itemize}
\item[$(\F_1)$] the value $F(t,x)\subset\R{N}$ is nonempty compact and convex for each $(t,x)\in I\times K$,
\item[$(\F_2)$] the multimap $I\ni t\mapsto F(t,x)$ possesses a measurable selection for all $x\in K$,
\item[$(\F_3)$] the multimap $K\ni x\mapsto F(t,x)$ is upper semicontinuous for a.a. $t\in I$,
\item[$(\F_4)$] $F$ is bounded, i.e. there exists $c>0$ such that for a.a. $t\in I$ and all $x\in K$, $\sup\limits_{y\in F(t,x)}|y|\<c$.
\end{itemize}
By a solution to problem \eqref{inclusion} we mean an absolutely continuous function $x\colon I\to\R{N}$ such that $x(0)=g(x)$, $x(I)\subset K$ and $\dot{x}(t)\in F(t,x(t))$ for a.a. $t\in I$.

\section{Existence theorems for nonlocal Cauchy problems}
The first assertion illustrates the relationship between the existence of fixed points of the Poincar\'e-like operator associated with nonlocal Cauchy problem \eqref{inclusion} and the presumption that the set of state constraints $K$ is homotopy dominated by the space $C(I,K)$ through the boundary operator $g$. The symbol ``$\simeq$'' denotes the relation of being homotopic.
\begin{theorem}\label{th9}
Assume that $K$, represented by $f$, is compact strictly regular with nontrivial Euler characteristic $\chi(K)\neq 0$. Suppose that $F\colon I\times K\map\R{N}$ satisfies conditions $(\F_1)$-$(\F_4)$ and 
\begin{equation}\label{f}
\forall\,x\in\bd K\text{ the multimap }F(\cdot,x)\cap\partial f(x)^\circ\text{ possesses a measurable selection.}
\end{equation}
Let $g\colon C(I,\R{N})\to\R{N}$ be a continuous mapping such that $g(C(I,K))\subset K$ and $g\circ\res{i}{K}\simeq id_K$. Then the nonlocal initial value problem \eqref{inclusion} possesses at least one solution.
\end{theorem}
\begin{proof}
Define so-called solution set map $S_F\colon K\map C(I,K)$, associated with the Cauchy problem 
\begin{equation}\label{cauchy}
\begin{gathered}
x'\in F(t,x),\;t\in I,x\in K\\
x(0)=x_0\in K,
\end{gathered}
\end{equation}
by the formula \[S_F(x_0):=\left\{x\in C(I,K)\colon x\mbox{ is a solution of }\eqref{cauchy}\right\}.\] In view of \cite[Cor.1.12.]{bader} $S_{\!F}$ is a set-valued map with compact $R_\delta$ values. Therefore, it is an acyclic multimap. Now we are able to introduce the Poincar\'e-like operator $P\colon K\map K$ related to problem \eqref{inclusion}, given by $P=g\circ S_F$. From what we have established so far it follows that $P$ is an admissible map.\par Observe that $g\circ\res{i}{K}\circ ev_0\simeq g$ and the joining homotopy $h\colon[0,1]\times C(I,K)\to K$ is given by $h(\lambdaup,x):=g(\pi(\lambdaup T,x))$, where $\pi\colon I\times C(I,K)\to C(I,K)$ is such that
\[\pi(t,x)(s):=
\begin{cases}
x(s),&s\in[0,t]\\
x(t),&s\in[t,T].
\end{cases}\]
By the assumption, we see that $g\circ\res{i}{K}\circ ev_0\simeq id_K\circ ev_0=ev_0$. This means that there is a homotopy $\ell\colon[0,1]\times C(I,K)\to K$ joining the evaluation $ev_0$ with the boundary operator $g$. Let $H\colon[0,1]\times K\map[0,1]\times C(I,K)$ be an acyclic multimap given by $H(\lambdaup,x_0):=\{\lambdaup\}\times S_{\!F}(x_0)$. Now, let us define an admissible homotopy $\psi\colon[0,1]\times K\map K$ by the following formula $\psi:=\ell\circ H$. Since $K$ is an ANR (\cite[Prop.3.1.]{bader}) and, of course, $\psi(0,x)=x$ on $K$, Theorem \ref{homo}. implies that there is $x_0\in K$ such that $x_0\in\psi(1,x_0)=P(x_0)$. This point corresponds to the solution $x\in C(I,K)$ of the nonlocal Cauchy problem \eqref{inclusion} in such a way that $x(0)=x_0=g(x)$.
\end{proof}
\begin{remark}
Suppose all the assumptions of Theorem \ref{th9}. are satisfied. Then the set $S_{\!F}(g)$ of solutions to the problem \eqref{inclusion} is nonempty compact.
\end{remark}
\begin{proof}
Retaining the notation of the proof of Theorem \ref{th9}., observe that $\overline{\fix(P)}=\fix(P)\subset K$. Thus, $\fix(P)$ is compact. The set-valued map $\Psi\colon\fix(P)\map C(I,K)$, given by the formula \[\Psi(x_0):=S_F(x_0)\cap g^{-1}(\{x_0\}),\] is a compact valued upper semicontinuous map. Since $S_{\!F}(g)$ coincides with $\Psi(\fix(P))$, the solution set $S_{\!F}(g)$ must be compact.
\end{proof}
Writing that the set K is symmetric we mean that it is symmetric with respect to origin, i.e. $K=-K$.
\begin{corollary}\label{contractible}
Assume that $K$, represented by $f$, is symmetric compact contractible and strictly regular. Suppose that $F\colon I\times K\map\R{N}$ satisfies conditions $(\F_1)$-$(\F_4)$ and \eqref{f}. Then the antiperiodic problem 
\begin{equation}\label{anti}
\begin{gathered}
x'\in F(t,x),\;t\in I,x\in K\\
x(0)=-x(T)
\end{gathered}
\end{equation}
has at least one solution.
\end{corollary}
\begin{proof}
Notice that $\chi(K)=1$. In order to apply Theorem \ref{th9}. it suffices to show that $(-ev_T)\circ\res{i}{K}\simeq id_K$. However, $(-ev_T)\circ\res{i}{K}=-id_K$. On the other hand $id_K\simeq -id_K$, because $K$ is symmetric and contractible.
\end{proof}
\begin{corollary}
Assume that $K$, represented by $f$, is compact and convex. Let $0<t_1<t_2<\ldots<t_n\<T$ be arbitrary, but fixed, $\alpha_i\geqslant 0$ and $\sum_{i=1}^n\alpha_i=1$. Suppose that $F\colon I\times K\map\R{N}$ satisfies conditions $(\F_1)$-$(\F_4)$ and \eqref{f}.  Then the following nonlocal initial value problem with multi-point discrete mean condition
\begin{equation}\label{multi2}
\begin{gathered}
x'\in F(t,x),\;t\in I,x\in K\\
x(0)=\sum\limits_{i=1}^n\alpha_ix(t_i)
\end{gathered}
\end{equation}
has at least one solution.
\end{corollary}
\begin{proof}
The constraint set $K$ is strictly regular with nontrivial characteristic. Let $g:=\sum\limits_{i=1}^n\alpha_iev_{t_i}$. Then $g(C(I,K))\subset K$ and $g\circ\res{i}{K}=id_K$.
\end{proof}
\begin{corollary}
Let $h\colon\R{N}\to\R{N}$ be a continuous mapping such that $|h(x)|\<L|x|$ for some $L>0$ and every $x\in\R{N}$. Assume that $K$, represented by $f$, is a compact convex set, invariant under the mapping $h$. Let $F\colon I\times K\map\R{N}$ satisfy conditions $(\F_1)$-$(\F_4)$ and \eqref{f}. Then the following nonlocal Cauchy problem with mean value condition
\begin{equation}\label{mean2}
\begin{gathered}
x'\in F(t,x),\;t\in I,x\in K\\
x(0)=\frac{1}{T}\int_0^Th(x(t))\,dt
\end{gathered}
\end{equation}
has at least one solution.
\end{corollary}
\begin{proof}
Define a continuous operator $g\colon C(I,\R{N})\to\R{N}$, by $g(x):=(1/T)\int_0^Th(x(t))\,dt$. The mean value theorem ensures that $g(x)\in\overline{\co}\,h(K)\subset\overline{\co}K=K$ for $x\in C(I,K)$. Observe that $g\circ\res{i}{K}=\res{h}{K}$. Since $K$ is convex, one sees that $\res{h}{K}\simeq id_K$. Eventually, $g\circ\res{i}{K}\simeq id_K$.
\end{proof}
Indication of the appropriate open, relatively compact subset $\Omega$, which is necessary for the application of continuation theorem (Theorem \ref{maw}.), will become possible by the use of the following notion (formulated below geometrical conditions generalize in a natural way those contained in \cite[Def.3.2.]{gaines}).
\begin{definition}\label{boundset}
An open bounded subset $K\subset\R{N}$ will be called an $($autonomous$)$ bound set relative to the equation $\dot{x}(t)=h(t,x(t))$ on $[0,T]$ if for any $x_0\in\bd K$ there exists $f(\cdot;x_0)=f(\cdot)$ such that
\begin{itemize}
\item[(i)] $f\colon\dom(f)\subset\R{N}\to\R{}$ is locally Lipschitz,
\item[(ii)] $f(x_0)=0$,
\item[(iii)] $\overline{K}\subset\{x\in\dom(f)\colon f(x)\<0\}$,
\item[(iv)] $\forall\,t\in(0,T)\;\forall\,p\in\partial f(x_0)\;\;\;\langle p,h(t,x_0)\rangle\neq 0.$
\end{itemize}
\end{definition}
The interest of Definition \ref{boundset}. follows from the subsequent observation.
\begin{proposition}\label{boundsetprop}
Let $K$ be a bound set relative to the equation $\dot{x}(t)=h(t,x(t))$ on $[0,T]$. If $x\in C\left([0,T],\overline{K}\right)$ is a solution to 
\begin{equation}\label{boundeq}
\dot{x}(t)=\lambdaup\,h(t,x(t)),
\end{equation}
with $\lambdaup\in(0,1)$, then $x((0,T))\subset K$.
\end{proposition}
\begin{proof}
Let $x$ be a solution to \eqref{boundeq} such that $x(t)\in\overline{K}$ for $t\in[0,T]$. Suppose that $x(t_0)\in\bd K$ for some $t_0\in(0,T)$. Let $f(\cdot)=f(\cdot;x(t_0))$ be the function given by Definition \ref{boundset}. It follows from properties (i)-(iii) that $f\circ x\colon(0,T)\to\R{}$ is locally Lipschitz and has a maximum at $t_0$. Therefore, $0\in\partial(f\circ x)(t_0)$ by \cite[Prop.2.3.2]{clarke}. Applying a chain rule \cite[Th.2.3.10]{clarke} we see that $\partial(f\circ x)(t_0)\subset\{\langle p,\dot{x}(t_0)\rangle\colon p\in\partial f(x(t_0))\}$. On the other hand, for every $p\in\partial f(x(t_0))$, $\langle p,\dot{x}(t_0)\rangle=\langle p,\lambdaup\, h(t_0,x(t_0))\rangle\neq 0$, by (iv) - a contradiction.
\end{proof}
We shall make use of the following topological description of strongly regular sets as Polish metric spaces.
\begin{lemma}\label{separable}
If the set $K$, represented by $f$, is strongly regular, then $\Int K\neq\varnothing$ and $K=\overline{\Int K}$. In particular, $K$ forms a separable subspace of the Euclidean space $\R{N}$.
\end{lemma}
\begin{proof}
Fix $x\in\bd K$. Suppose that there is $\eps>0$ such that $f(y)\geqslant 0$ for all $y\in B(x,\eps)$. From Definition \ref{K}. follows that $\inf\limits_{p\in\partial f(x)}|p|>0$, i.e.
\begin{align*}
\inf_{p\in\partial f(x)}\sup_{|v|=1}\,\langle p,v\rangle>0\Leftrightarrow\sup_{|z|=1}\inf_{p\in\partial f(x)}\langle p,v\rangle>0\Leftrightarrow-\inf_{|z|=1}\sup_{p\in\partial f(x)}\langle p,v\rangle>0\Leftrightarrow\inf_{|z|=1}f^\circ(x;v)<0.
\end{align*} 
Thus, there exists $v_x\in\R{N}$ such that $|v_x|=1$ and $f^\circ(x;v_x)<0$. For $h\in(0,\eps)$ we have $x+hv_x\in B(x,\eps)$ and by supposition $f(x+hv_x)\geqslant 0$. Passing to the limit with $h$ we obtain:
\[0\<\underset{h\to 0^+}{\overline{\lim}}\;\frac{f(x+hv_x)-f(x)}{h}\<\underset{h\to 0^+}{\underset{y\to x}{\overline{\lim}}}\frac{f(y+hv_x)-f(y)}{h}=f^\circ(x;v_x)<0\] - a contradiction. Thus, in every neighbourhood of the point $x\in\bd K$ we will find such a point $y$ that $f(y)<0$, i.e. $y\in\Int K$. Hence the conclusion that $K\subset\overline{\Int K}$.
\end{proof}
In the next two results we formulate sufficient conditions for the existence of antiperiodic solutions of differential inclusions defined on strongly regular sets. The proof of the first of them takes advantage of the continuation theorem. Theorem \ref{borsuk1}. is a consequence of the application of the generalized Borsuk theorem. In what follows $\ev_t\colon C(I,\R{N})\to\R{N}$ stands for the evaluation at point $t\in I$.
\begin{theorem}\label{th1}
Let $K$, represented by $f$, be a compact and strongly regular set. Assume that $F\colon I\times K\map\R{N}$ fulfills $(\F_1)$-$(\F_4)$ and that either \eqref{f} is satisfied or
\begin{equation}\label{-f}
\forall\,x\in\bd K\;\text{ the multimap }F(\cdot,x)\cap\partial(-f)(x)^\circ\text{ has a measurable selection}.
\end{equation}
Let $g\colon\R{N}\times\R{N}\to\R{N}$ be a continuous map. If the following conditions hold
\begin{equation}\label{warprop2}
\begin{gathered}
0\not\in g(\cdot,\cdot)(\bd K)\\
\deg\left(g(\cdot,\cdot),\Int K,0\right)\neq 0
\end{gathered}
\end{equation}
and
\begin{equation}\label{boundarycond}
\begin{aligned}
&\text{for every }x\in C^1(I,K)\cap(g\circ(\ev_0\times\ev_T))^{-1}(0)\text{ such that }x((0,T))\cap\bd K=\varnothing\\&\text{there is equivalence }x(0)\in\bd K\Leftrightarrow x(T)\in\bd K
\end{aligned}
\end{equation}
then the following nonlocal boundary value problem 
\begin{equation}\label{gcdot}
\begin{gathered}
x'\in F(t,x),\;\text{a.e. on } I,x\in K\\
g(x(0),x(T))=0
\end{gathered}
\end{equation}
possesses at least one solution.
\end{theorem}
\begin{proof}
Let $h\in C(I\times\R{N},\R{N})$. Assume that either 
\begin{equation}\label{h}
\forall\,x_0\in\bd K\;\forall\,t\in I\;\;\;f^\circ(x_0;h(t,x_0))<0
\end{equation}
or 
\begin{equation}\label{-h}
\forall\,x_0\in\bd K\;\forall\,t\in I\;\;\;f^\circ(x_0;-h(t,x_0))<0.
\end{equation}
We claim that
\begin{equation}\label{claim1}
\begin{aligned}
&\text{under condition \eqref{h} every solution }x\in C(I,K)\text{ of the differential equation }\\&\dot{x}(t)=\lambdaup\,h(t,x(t))\text{ on }I,\text{with }\lambdaup\in(0,1),\text{satisfies }x(T)\not\in\bd K
\end{aligned}
\end{equation}
and
\begin{equation}\label{claim2}
\begin{aligned}
&\text{under condition \eqref{-h} every solution }x\in C(I,K)\text{ of the differential equation }\\&\dot{x}(t)=\lambdaup\,h(t,x(t))\text{ on }I,\text{with }\lambdaup\in(0,1),\text{satisfies }x(0)\not\in\bd K.
\end{aligned}
\end{equation}
Let $x$ be a solution to $\dot{x}(t)=\lambdaup\,h(t,x(t))$ on $I$ such that $x(t)\in K$ for $t\in I$. Denote by $\Omega_f$ the set of measure zero, where $f$ is not differentiable. There is a subset $J$ of full measure in the segment $I$ and a sequence $(z_n)_\n$ convergent to $0\in\R{N}$ such that $x_n(t):=x(t)+z_n\not\in\Omega_f$ for every $t\in J$ and $\n$. A similar line of reasoning was used in \cite[Lemma.3.]{lewicka}. Observe that $\{x_n(I)\}_\n$ is relatively compact and $f'\left(\{x_n(J)\}_\n\right)$ is a subset of a compact image $\partial f\left(\overline{\{x_n(I)\}}_\n\right)$. Therefore, there is $M>0$ such that $|f'(x_n(t))|\<M$ for $t\in J$ and $\n$. Consequently, $\{(f\circ x_n)(t)\}_\n$ is relatively compact for $t\in I$ and $|(f\circ x_n)'(t)|=|\langle f'(x_n(t)),\dot{x}_n(t)\rangle|\<M|\dot{x}(t)|$ for $t\in J$ and $\n$. Hence, by passing to a subsequence if necessary, we may assume that $f\circ x_n\to f\circ x$ uniformly on $I$ and $(f\circ x_n)'\to(f\circ x)'$ weakly in $L^1(I,\R{N})$ (see \cite[Th.0.3.4]{aubin}). By Mazur's lemma there is a sequence $(w_k)_{k\geqslant 1}$ strongly convergent to $(f\circ x)'$ such that $w_k\in\co\{(f\circ x_n)'\}_{n=k}^\infty$. We can assume w.l.o.g. that $(f\circ x)'(t)=\lim_{k\to\infty}w_k(t)$ for $t\in J$.\par Suppose for definiteness that $x(0)\in\bd K$. Then $f^\circ(x(0);-h(0,x(0)))<0$, due to the assumption \eqref{-h}. Since the mapping $f^\circ(\cdot;\cdot)$ is upper semicontinuous, there exixsts $\delta>0$ such that for every $t\in(0,\delta)$ we have $f^\circ(x(t);-\lambdaup\,h(t,x(t)))<0$.\par Fix $t_0\in(0,\delta)\cap J$. Let $\eta_{t_0}>0$ be such that $f^\circ(x(t_0);-\lambdaup\,h(t_0,x(t_0)))<-\eta_{t_0}$. One should be aware that $f^\circ(x(t_0);-\dot{x}(t_0))=\inf_{p\in\partial f(x(t_0))}\langle p,\dot{x}(t_0)\rangle>0$, which entails $|\dot{x}(t_0)|>0$. Take an arbitrary $\eps>0$. Due to upper semicontinuity of the generalized gradient $\partial f$, there must be a number $n_0$ such that for every $n\geqslant n_0$
\[-(f\circ x_n)'(t_0)=\langle f'(x(t_0)+z_n),-\dot{x}(t_0)\rangle\in\left\{\langle p+u,-\dot{x}(t_0)\rangle\colon p\in\partial f(x(t_0))\text{ and }|u|<\frac{\eps}{|\dot{x}(t_0)|}\right\}.\] For that reason $-(f\circ x_n)'(t_0)<f^\circ(x(t_0);-\dot{x}(t_0))+\eps<-\eta_{t_0}+\eps$ for $n\geqslant n_0$. Consequently, $w_k(t_0)>\eta_{t_0}-\eps$ for $k$ large enough and eventually $(f\circ x)'(t_0)\geqslant\eta_{t_0}$. As a matter of fact, we have shown that $(f\circ x)'(t)>0$ a.e. on $(0,\delta)$.\par Since $f\circ x\colon[0,\delta]\to\R{}$ is absolutely continuous (for suitably chosen $\delta$), we conclude that
\[f(x(\delta))-f(x(0))=\int_0^\delta(f\circ x)'(s)\,ds>0,\] i.e. $f(x(0))<f(x(\delta))\<0$. Thus, $x(0)\in\Int K$ and we arrive at contradiction.\par Now, assume for definiteness that condition \eqref{-f} holds. Reasoning based on the assumption \eqref{f} proceeds in an analogous manner. Construction given in the proof of \cite[Th.1.6.]{bader} indicates that for every $x\in\bd K$ and any positive integer $n$ there exists a measurable $v_x^n\colon I\to\R{N}$ such that 
\begin{equation}\label{bound}
v_x^n(t)\in B\left(F(t,x),\frac{1}{n}\right)\;\;\text{a.e. on }I
\end{equation}
and
\begin{equation}\label{<}
(-f)^\circ(y;v_x^n(t))<0\;\;\text{for every }(t,y)\in I\times B\left(x,\delta_x^n\right)\;\text{with }\delta^n_x\in\left(0,\frac{1}{3^n}\right).
\end{equation}
Let us mention in the context of the aforesaid proof that strong regularity implies both $\inf\limits_{|u|=1}f^\circ(x;u)<0$ and $\inf\limits_{|u|=1}(-f)^\circ(x;u)<0$ for $x\in\bd K$. If $x\in\Int K$, then we can choose a radius $\delta^n_x\in\left(0,\frac{1}{3^n}\right)$ and a measurable $v_x^n\colon I\to\R{N}$ in such a way that $B\left(x,\delta_x^n\right)\subset\Int K$ and $v_x^n(t)\in F(t,x)$ a.e. on $I$. \par For any $\n$ we have constructed a relatively open covering $\left\{B\left(x,\delta^n_x\right)\cap K\right\}_{x\in K}$ of $K$. Clearly, there exists a locally Lipschitz partition of unity $\left\{\lambdaup_y^n\colon K\to[0,1]\right\}_{y\in K}$ such that the family of supports $\left\{\supp\lambdaup_y^n\right\}_{y\in K}$ forms a locally finite (closed) covering of the space $K$, inscribed into the covering $\left\{B\left(y,\delta^n_y\right)\cap K\right\}_{y\in K}$, i.e. $\supp\lambdaup_y^n\subset B\left(y,\delta_y^n\right)$ for $y\in K$.\par Now, let us define a map $W_n\colon I\times K\to\R{N}$ by the formula 
\begin{equation}\label{W_n}
W_n(t,x):=\sum_{y\in K}\lambdaup^n_y(x)v^n_y(t).
\end{equation}
Of course, $W_n$ is of Carath\'eodory type. Since $K$ is a separable space (Lemma \ref{separable}.), the function $W_n$ possesses the Scorza-Dragoni's property (see \cite[Th.1.]{ricceri}). Therefore, given $\K$, one may find a closed $I_k\subset I$ such that the Lebesgue measure $\ell(I\setminus I_k)<\frac{1}{k}$ and the restriction $\res{W_n}{I_k\times K}$ is continuous.\par Let $\{I_s,t_s\}_{s\in S}$ be a Dugundji system for the complement $I_k^c$, while $\{\mu^n_{k,s}\}_{s\in S}$ be a family of continuous functions $\mu^n_{k,s}\colon I_k^c\to[0,1]$ constituting a locally finite partition of unity subordinated to the covering $\{I_s\}_{s\in S}$. An extension of $\res{W_n}{I_k\times K}$ to the continuous function $W_n^k\colon I\times K\to\R{N}$ can be achieved routinely by the formula
\begin{equation}\label{W_n^k}
W_n^k(t,x):=
\begin{cases}
W_n(t,x)&\text{for }(t,x)\in I_k\times K,\\
\sum\limits_{s\in S}\mu_{k,s}^n(t)W_n(t_s,x)&\text{for }(t,x)\in I_k^c\times K.
\end{cases}
\end{equation}
From \eqref{bound} we infer that every mapping $W_n^k$ is bounded by $c+\frac{1}{n}$ on the whole product $I\times K$. If $x\in\bd K\cap\supp\lambdaup^n_y$, then $y\in\bd K$. Otherwise $x\in\supp\lambdaup^n_y\subset B(y,\delta^n_y)\subset\Int K$ - contradiction. Hence, for $x\in\bd K$ and $S(x):=\{y\in K\colon x\in\supp\lambdaup^n_y\}$ we have, for $t\in I$
\begin{equation}\label{<W_n^k}
f^\circ\left(x;-W_n^k(t,x)\right)=(-f)^\circ\left(x;W_n^k(t,x)\right)\<\sum_{s\in S}\mu_{k,s}^n(t)\sum_{y\in S(x)}\lambdaup^n_y(x)\,(-f)^\circ\left(x;v^n_y(t_s)\right)<0,
\end{equation}
by \eqref{<}, \eqref{W_n} and \eqref{W_n^k}.\par Put $\Omega:=C(I,\Int K)$. Suppose that $x\in\overline{\Omega}$ is a solution to the following nonlocal boundary value problem
\[\begin{gathered}
\dot{x}(t)=\lambdaup W_n^k(t,x(t)),\;t\in I,\\
g(x(0),x(T))=0,
\end{gathered}\]
for some $\lambdaup\in(0,1)$. Taking into account property \eqref{<W_n^k} one easily sees that $\Int K$ is a bound set relative to the equation $\dot{x}(t)=W_n^k(t,x(t))$ on $I$. In view of Proposition \ref{boundsetprop}., $x((0,T))\cap\bd K=\varnothing$. Furthermore, by claim \eqref{claim2} we infer that $x(0)\not\in\bd K$. Therefore, $x(t)\not\in\bd K$ for each $t\in I$, due to the assumption \eqref{boundarycond}. In fact, we have shown that $x\not\in\bd\Omega$.\par Let $\gamma:=g\circ(\ev_0\times\ev_T)$. The equivalence between assumptions (ii) in Theorem \ref{maw}. and $0\not\in g(\cdot,\cdot)(\bd K)$ is evident. Put $\w{\Omega}:=\{x_0\in\R{N}\colon i(x_0)\in\Omega\}$. Clearly, $\w{\Omega}=\Int K$ and \[\deg(\gamma\circ i,\w{\Omega},0)=\deg\left(g(\cdot,\cdot),\Int K,0\right).\] The last quantity is nonzero, by \eqref{warprop2}. Applying Theorem \ref{maw}. we obtain a solution $x_k\in\overline{\Omega}$ of the undermentioned boundary value problem
\begin{equation}
\begin{gathered}
\dot{x}(t)=W_n^k(t,x(t)),\;t\in I,\\
g(x(0),x(T))=0.
\end{gathered}
\end{equation}
\par In view of compactness theorem \cite[Th.0.3.4.]{aubin} there is an accumulation point $\w{x}\in\overline{\Omega}$ of the sequence of solutions $(x_k)_\K$. We claim that $\w{x}$ is a Carath\'eodory solution of the following problem  
\begin{gather}
\dot{x}(t)=W_n(t,x(t)),\;\text{a.e. on }I,\label{equation}\\
g(x(0),x(T))=0.\label{condition}
\end{gather}
Indeed, for each $t\in I$ we are dealing with the following estimation
\begin{align*}
|\int_0^tW_n^k&(s,x_k(s))\,ds-\int_0^tW_n(s,\w{x}(s))\,ds|\\&\<\int_{[0,t]\cap I_k}\left|W_n^k(s,x_k(s))-W_n(s,\w{x}(s))\right|\,ds+\int_{[0,t]\cap I_k^c}\left|W_n^k(s,x_k(s))-W_n(s,\w{x}(s))\right|\,ds\\&\<\int_{[0,t]\cap I_k}\left|W_n(s,x_k(s))-W_n(s,\w{x}(s))\right|\,ds+2\int_{[0,t]\cap I_k^c}\left(c+\frac{1}{n}\right)\,ds.
\end{align*}
The first of the last two expressions tends to zero by virtue of Lebesgue's dominated convergence theorem and the other is not greater than $2\left(c+\frac{1}{n}\right)\frac{1}{k}$. This shows that $\w{x}$ satisfies equation \eqref{equation}. Obviously, it also satisfies boundary condition \eqref{condition}, due to the continuity of the operator $g$.\par As we have established above, for each $\n$, there exists a solution $x_n\in\overline{\Omega}$ of the problem \eqref{equation}-\eqref{condition}. One easily sees that \[W_n(t,x)\in\sum_{y\in K}\lambdaup_y^n(x)\,\overline{\co}B\left(F(t,B(y,2\cdot3^{-n})\cap K),\frac{1}{n}\right)\subset\overline{\co}B\left(F\left(t,B\left(x,3^{-n+1}\right)\cap K\right),\frac{1}{n}\right)\] for every $\n$, $x\in K$ and for a.a. $t\in I$. According to theorem \cite[Th.0.3.4.]{aubin}, the sequence $(x_n)_\n$ possesses an accumulation point $x\in AC(I,K)$, such that
\[\begin{cases}
x_n(t)\to x(t)&\text{for }t\in I,\\
\dot{x}_n\rightharpoonup\dot{x}&\text{in }L^1(I,\R{N}),\\
\dot{x}_n(t)\in\overline{\co}B\left(F\left(t,B\left(x_n(t),3^{-n+1}\right)\cap K\right),\frac{1}{n}\right)&\text{a.e. on }I.
\end{cases}\]
Hence, by the convergence theorem (Theorem \ref{convergence}.), $\dot{x}(t)\in F(t,x(t))$ a.e. on $I$. Summing up, $x$ is the sought solution to the problem \eqref{inclusion}.
\end{proof}
\begin{corollary}
Assume that $K$ is a compact epi-Lipschitz subset of $\R{N}$. Let $g\colon\R{N}\times\R{N}\to\R{N}$ be a continuous map such that \eqref{warprop2} and \eqref{boundarycond} are fulfilled. If $F\colon I\times K\map\R{N}$ satisfies conditions $(\F_1)$-$(\F_4)$ and either
\begin{equation}\label{fepi}
\text{the multimap }F(\cdot,x)\cap C_K(x)\text{ has a measurable selection for every }x\in\bd K
\end{equation}
or
\begin{equation}\label{-fepi}
\text{the multimap }F(\cdot,x)\cap C_{K^c}(x)\text{ has a measurable selection for every }x\in\bd K,
\end{equation}
then the nonlocal boundary value problem \eqref{gcdot} possesses a solution.
\end{corollary}
\begin{proof}
An epi-Lipschitz set $K\subset\R{N}$ is represented by a Lipschitz continuous function $\Delta_K\colon\R{N}\to\R{}$ given by $\Delta_K:=d_K-d_{K^c}$. It turns out that $K$ is strongly regular with respect to $\Delta_K$ (\cite{cornet}). Furthermore, in view of \cite[Th.3.]{urruty} we have $\partial\Delta_K(x)^\circ=C_K(x)$. A simple calculation shows that the complementary set $K^c$ of $K$ in $\R{N}$ is also epi-Lipschitz and $\partial(-\Delta_K)(x)^\circ=\partial\Delta_{K^c}(x)^\circ=C_{K^c}(x)$.
\end{proof}

\begin{theorem}\label{borsuk1}
Suppose $K$, represented by an even $f$, is compact symmetric strongly regular and $0\in\Int K$. Let $F\colon I\times K\map\R{N}$ be a bounded lower Carath\'eodory type multivalued mapping with closed convex values. If, for every $x\in\bd K$, either condition
\begin{equation}\label{01}
0\in\left(\max\limits_{t\in I}\max\limits_{z\in F(t,x)}f^\circ(x;z),\min\limits_{t\in I}\min\limits_{u\in F(t,-x)}f_\circ(x;u)\right)
\end{equation}
or condition
\begin{equation}\label{02}
0\in\left(\max\limits_{t\in I}\max\limits_{u\in F(t,-x)}f^\circ(x;u),\min\limits_{t\in I}\min\limits_{z\in F(t,x)}f_\circ(x;z)\right)
\end{equation}
holds, then the antiperiodic problem \eqref{anti} has at least one solution.
\end{theorem}
\begin{proof}
In view of \cite[Th.3.2.]{prikry} the set-valued map $F$ admits a Carath\'eodory type selection $g\colon I\times K\to\R{N}$. The alternative of tangency conditions \eqref{01}-\eqref{02} placed in the context of the selection $g$ assumes the form: for every $x\in\bd K$ one of the following conditions holds
\begin{itemize}
\item[(a)] $f^\circ(x;g(t,x))<0<f_\circ(x;g(t,-x))$ for $t\in I$,
\item[(b)] $f^\circ(x;g(t,-x))<0<f_\circ(x;g(t,x))$ for $t\in I$.
\end{itemize}
As previously shown (that is, in the proof of Theorem \ref{th1}.) the {Ca\-ra\-th\'eodory} map $g$ is in fact uniformly $L_1$-{approxi\-mable} by certain $\left(g_k\right)_\K\subset C\left(I\times K,\R{N}\right)$, i.e. \[\int_I\sup_{x\in K}\left|\,g_k(t,x)-g(t,x)\right|\,dt\xrightarrow[k\to\infty]{}0.\] The approximant $g_k$ is defined by means of a suitable partition of unity $\left\{\mu_{k,s}\colon\! I_k^c\to[0,1]\right\}_{s\in S}$ associated with a Dugundji system $\{I_s,t_s\}_{s\in S}$ for the complement $I_k^c$ (see \eqref{W_n^k} for exact formula). It is easy to convince oneself that, for each $x\in\bd K$, functions $g_k$ also meet one of the hypotheses (a) or (b). Indeed, fix $x\in\bd K$. Suppose for definiteness that (a) holds. Then
\[f^\circ(x;g_k(t,x))=f^\circ\left(x;\sum_{s\in S}\mu_{k,s}(t)g(t_s,x)\right)\<\sum_{s\in S}\mu_{k,s}(t)f^\circ(x;g(t_s,x))<0\] and
\[f_\circ(x;g_k(t,-x))=-(-f)^\circ\left(x;\sum_{s\in S}\mu_{k,s}(t)g(t_s,-x)\right)\geqslant\sum_{s\in S}\mu_{k,s}(t)f_\circ(x;g(t_s,-x))>0\] for all $t\in I$.\par Let $X:=C(I,\R{N})$, $Z:=C(I,\R{N})\times\R{N}$, $\dom L:=C^1(I,\R{N})$; $L\colon\dom L\subset X\to Z$ be such that $Lx:=(\dot{x},0)$. Then $\ker L=i(\R{N})$, $\im L=C(I,\R{N})\times \{0\}$ and $\coker L\approx\R{N}$, i.e. $L$ is a Fredholm mapping of index zero. Define continuous linear operators $P\colon X\to X$ and $Q\colon Z\to Z$ by $P:=\ev_0$ and $Q((y,x_0)):=(0,x_0)$. Clearly, $(P,Q)$ is an exact pair of idempotent projections with respect to the mapping $L$. Let $\Omega:=C(I,\Int K)$ and $N\colon\overline{\Omega}\to Z$ be a nonlinear operator given by \[N(x):=(g_k(\cdot,x(\cdot)),\ev_0(x)+\ev_T(x)).\] It is obvious that the antiperiodic problem 
\begin{equation}\label{g_k}
\begin{gathered}
\dot{x}(t)=g_k(t,x(t)),\;t\in I,x\in K\\
x(0)=-x(T),
\end{gathered}
\end{equation}
is equivalent to the operator equation $(L-N)(x)=0$. \par Assume that $x\in\overline{\Omega}$ is a solution to the following nonlocal boundary value problem
\begin{equation}\label{g_klambda}
\begin{gathered}
(1+\lambdaup)\dot{x}(t)=g_k(t,x(t))-\lambdaup\,g_k(t,-x(t)),\;t\in I,x\in K\\
x(0)=-x(T),
\end{gathered}
\end{equation}
for some $\lambdaup\in(0,1]$. Suppose there is a point $t_0\in(0,T)$ such that $x(t_0)\in\bd K$. Then one of the conditions (a)-(b) is satisfied. Assuming for definiteness that (a) holds we obtain \[\langle p,g_k(t_0,x(t_0))\rangle<0<\langle p,g_k(t_0,-x(t_0))\rangle\] for all $p\in\partial f(x(t_0))$. Thus \[\langle p,g_k(t_0,x(t_0))-\lambdaup\,g_k(t_0,-x(t_0))\rangle<0,\] which means that $\langle p,(1+\lambdaup)\dot{x}(t_0)\rangle<0$ for all $p\in\partial f(x(t_0))$. Consequently, \[0\not\in\langle\partial f(x(t_0)),\dot{x}(t_0)\rangle=\{\langle p,\dot{x}(t_0)\rangle\colon p\in\partial f(x(t_0))\}.\] On the other hand, the composite function $f\circ x\colon(0,T)\to\R{}$ is locally Lipschitz and attains a maximum at $t_0$. On the basis of \cite[Prop.2.3.2]{clarke} and \cite[Th.2.3.10]{clarke} we infer that $0\in\partial(f\circ x)(t_0)\subset\langle\partial f(x(t_0)),\dot{x}(t_0)\rangle$, which yields a contradiction.\par Let $h\in C\left(I\times K,\R{N}\right)$ be given by $h(t,x):=g_k(t,x)-\lambdaup\,g_k(t,-x)$. The solution to the problem \eqref{g_klambda} is also a solution of the differential equation $\dot{x}(t)=(1+\lambdaup)^{-1}h(t,x(t))$ on $I$, where $(1+\lambdaup)^{-1}\in\left[\frac{1}{2},1\right)$.\par We claim that $x(0)\not\in\bd K$. Suppose to the contrary that $x(0)\in\bd K$. Moreover, assume that
\begin{equation}\label{first}
f^\circ(x(0);g_k(0,x(0)))<0<f_\circ(x(0);g_k(0,-x(0)))
\end{equation}
and at the same time condition
\begin{equation}\label{second}
f^\circ(x(T);g_k(T,-x(T)))<0<f_\circ(x(T);g_k(T,x(T)))
\end{equation}
is satisfied. Since $f$ is even, so by \eqref{first} we get
\begin{align*}
f^\circ(-x(T);g_k(0,-x(T)))<&0<f_\circ(-x(T);g_k(0,x(T)))\Leftrightarrow\\&\Leftrightarrow f^\circ(x(T);-g_k(0,-x(T)))<0<f_\circ(x(T);-g_k(0,x(T)))
\\&\Leftrightarrow f^\circ(x(T);g_k(0,x(T)))<0<f_\circ(x(T);g_k(0,-x(T))).
\end{align*}
The latter in connection with \eqref{second} entails
\[f^\circ(x(T);g_k(0,x(T)))<0<f^\circ(x(T);g_k(T,x(T))).\]
Given that $f^\circ(x(T);g_k(\cdot,x(T)))$ is a continuous map, there exists $t_0\in(0,T)$ such that $f^\circ(x(T);g_k(t_0,x(T)))=0$. However, since $x(T)\in\bd K$ (recall that $K$ is symmetric), we have $f^\circ(x(T);g_k(t_0,x(T)))<0$ or $f_\circ(x(T);g_k(t_0,x(T)))>0$ in view of (a)-(b). In both cases, there is $f^\circ(x(T);g_k(t_0,x(T)))\neq 0$ - a contradiction. Therefore, conditions \eqref{first} and \eqref{second} can not occur simultaneously. \par Suppose that condition \eqref{first} is not met. However, since $x(0)\in\bd K$, so the inequality $f^\circ(x(0);g_k(0,-x(0)))<0<f_\circ(x(0);g_k(0,x(0)))$ must be true. This means that \[\forall\,p\in\partial f(x(0))\;\;\;\langle p,-g_k(0,x(0))\rangle<0\,\wedge\,0<\langle p,-g_k(0,-x(0))\rangle.\] Thus \[\forall\,p\in\partial f(x(0))\;\;\;\langle p,-g_k(0,x(0))+\lambdaup\,g_k(0,-x(0))\rangle<0\] and eventually $f^\circ(x(0);-h(0,x(0)))<0$. At this point, we may refer to the proof of the property \eqref{claim2} to ascertain that $x(0)\not\in\bd K$.\par This time let us assume that condition \eqref{second} is not fulfilled. Thus $f^\circ(x(T);g_k(T,x(T)))<0<f_\circ(x(T);g_k(T,-x(T)))$ must hold, because $x(T)\in\bd K$. It follows that \[\forall\,p\in\partial f(x(T))\;\;\;\langle p,g_k(T,x(T))\rangle<0\,\wedge\,0<\langle p,g_k(T,-x(T))\rangle\] and thus \[\forall\,p\in\partial f(x(T))\;\;\;\langle p,g_k(T,x(T))-\lambdaup\,g_k(T,-x(T))\rangle<0.\] Consequently, $f^\circ(x(T);h(T,x(T)))<0$. As previously, referring to the justification of the property \eqref{claim2} and, indirectly, of the property \eqref{claim1} we are able to deduce that $x(T)\not\in\bd K$. Hence $x(0)\not\in\bd K$.\par Summarizing the findings so far, we see that $x(I)\cap\bd K=\varnothing$, i.e. $x\not\in\bd\Omega$. As can be easily seen, the boundary value problem \eqref{g_klambda} is equivalent to the following operator equation
\[(L-N)(x)=\lambdaup(L-N)(-x)\] in $\dom L\cap\overline{\Omega}$. Therefore, we have just proved that $(L-N)(x)\neq\lambdaup(L-N)(-x)$ for every $(x,\lambdaup)\in(\dom L\cap\bd\Omega)\times(0,1]$. By virtue of \cite[Th.IV.2.]{maw} there is at least one solution $x_k\in C^1(I,K)$ to the antiperiodic problem \eqref{g_k}. The compactness of $K$ implies that $(x_k)_\K$ is a bounded sequence in $C(I,\R{N})$. Since the solutions $x_k$ are equicontinuous, by passing to a subsequence if necessary, we may suppose that $x_k\to x_0$ uniformly on $I$, where $x_0\colon I\to\R{N}$ is absolutely continuous. By virtue of the following estimation
\begin{align*}
|\int_0^tg_k(s,x_k(s))\,ds&-\int_0^tg(s,x_0(s))\,ds|\\&\<\int_I\left|g_k(s,x_k(s))-g(s,x_k(s))\right|\,ds+\int_I\left|g(s,x_k(s))-g(s,x_0(s))\right|\,ds\\&\<\int_I\sup_{x\in K}\left|g_k(s,x)-g(s,x)\right|\,ds+\int_I\left|g(s,x_k(s))-g(s,x_0(s))\right|\,ds
\end{align*}
for $t\in I$, we infer that $x_0$ is a solution to the following antiperiodic problem 
\begin{equation*}
\begin{gathered}
\dot{x}(t)=g(t,x(t)),\;\text{a.e. on }I,x\in K\\
x(0)=-x(T).
\end{gathered}
\end{equation*}
Seeing that the function $g$ is a selection of the multivalued right-hand side $F$, it is absolutely clear that $x_0$ provides also a solution to the problem \eqref{anti}.
\end{proof}
\begin{remark}
It should be noted that Theorem \ref{borsuk1}. constitutes a distinct result in relation to Theorem \ref{th1}., even in the case when the right-hand side $F$ of the inclusion \eqref{anti} is univalent. Under the assumptions of Theorem \ref{borsuk1}. neither condition \eqref{f} nor \eqref{-f} is satisfied. For instance, assumption \eqref{01} is in fact equivalent with conjunction
\[\forall\,t\in I\;\;\;F(t,x)\subset\Int\left[\partial f(x)^\circ\right]\,\wedge\,F(t,-x)\subset\Int\left[\partial(-f)(x)^\circ\right].\] On the other hand, if the representing function $f$ is even, then, in particular, condition \eqref{f} implies \[\forall\,x\in\bd K\;\;\;0\in\left[\max\limits_{t\in I}\min\limits_{z\in F(t,x)}f^\circ(x;z),\min\limits_{t\in I}\max\limits_{u\in F(t,-x)}f_\circ(x;u)\right].\] 
\end{remark}
\par Let $K:=f^{-1}((-\infty,0])$ be a compact regular set. Definition \ref{K}. indicates that there is an open neighbourhood $\Omega$ with compact closure such that $K\subset\Omega\subset\overline{\Omega}\subset\dom(f)$ and $0\not\in\partial f(x)$ for each $x\in\overline{\Omega}\setminus K$. Let $M:=\inf_{x\in\bd\Omega}f(x)$ and $\Omega_M:=\{x\in\overline{\Omega}\colon f(x)<M\}$. For any $\eps\in(0,M)$ we may define $K_\eps:=\{x\in\overline{\Omega}\colon f(x)\<\eps\}\subset\Omega$. It is clear that the compact $\eps$-sublevel set $K_\eps$ is strongly regular as a constraint set represented by $\Omega\ni x\mapsto f_\eps(x):=f(x)-\eps$, i.e. $0\not\in\partial f_\eps(x)=\partial f(x)$ provided $x\in\bd K_\eps$. \par The subsequent theorem addresses the question of existence of trajectories forming antiperiodic solutions to the viability problem with the set of state constraints having possibly empty interior. The idea behind this result was to refine the outcome of Theorem \ref{th9}. in the context of antiperiodic boundary condition.
\begin{theorem}\label{ant}
Assume that $K$, represented by $f$, is a compact regular set containing origin. Suppose further that
\begin{equation}\label{epilevel}
\exists\,\eps_0>0\;\forall\,\eps\in(0,\eps_0)\;\;\;\dom(f)\setminus\Int K_\eps\text{ is symmetric.}
\end{equation}
Assume that $F\colon I\times K\map\R{N}$ satisfies $(\F_1)$-$(\F_4)$ and that either 
\begin{equation}\label{liminf}
\forall\,x\in\bd K\;\;\;F(\cdot,x)\cap\underset{y\to x,y\notin K}{\Liminf}\,\partial f(y)^\circ\text{ has a measurable selection}
\end{equation}
or
\begin{equation}\label{-liminf}
\forall\,x\in\bd K\;\;\;F(\cdot,x)\cap\underset{y\to x,y\notin K}{\Liminf}\,\partial(-f)(y)^\circ\text{ has a measurable selection.}
\end{equation}
Then there is a solution to the antiperiodic problem \eqref{anti}.
\end{theorem}
\begin{proof}
Making use of a Dugundji system $\{U_s,a_s\}_{s\in S}$ for the complement $\R{N}\setminus K$ one defines 
\[\w{F}(t,x):=
\begin{cases}
F(t,x)&\text{for }(t,x)\in I\times K,\\
\sum\limits_{s\in S}\mu_s(x)F(t,a_s)&\text{for }(t,x)\in I\times(\R{N}\setminus K),
\end{cases}\]
where $\{\mu_s\}_{s\in S}$ constitutes a locally finite partition of unity inscribed into the covering $\{U_s\}_{s\in S}$. This method of extension of the map $F$ ensures fulfillment of assumptions $(\F_1)$-$(\F_4)$ also at points outside the set $K$.\par Relying on locally Lipschitz partition of unity $\left\{\lambdaup_y^m\colon\R{N}\to[0,1]\right\}_{y\in\R{N}}$ subordinated to the covering $\{B(y,r_m)\}_{y\in\R{N}}$ of the Euclidean space $\R{N}$ we may introduce, for any $m\geqslant 1$, a set-valued approximation $F_m\colon I\times\R{N}\map\R{N}$ by the formula
\begin{equation}\label{Fmform}
F_m(t,x):=\sum_{y\in\R{N}}\lambdaup^m_y(x)\,\overline{\co}B\left(\w{F}(t,B(y,2r_m)),\frac{1}{m}\right),
\end{equation}
where $r_m:=3^{-m}$. Stands to reason, the maps $F_m$ also satisfy conditions $(\F_1)$-$(\F_4)$ (being bounded by a constant $c+\frac{1}{m}$). Multimaps $F_m$ approximate the values of the mapping $\w{F}$ in such a way that
\begin{equation}\label{F_m}
\w{F}(t,x)\subset F_{m+1}(t,x)\subset F_m(t,x)\subset\overline{\co}B\left(\w{F}(t,B(x,3r_m)),\frac{1}{m}\right)
\end{equation} 
for $m\geqslant 1$ and every $(t,x)\in I\times\R{N}$.
\par\noindent We claim that assumption \eqref{liminf} (assumption \eqref{-liminf}) implies
\begin{equation}\label{v_x,y}
\begin{aligned}
&\text{for any }m\geqslant 1\text{, there is }\delta_m>0\text{ such that, for any }y\in\bd K\text{ and }x\in B(y,\delta_m)\setminus K,\\&\text{there is a measurable }v_{y,x}\colon I\to\R{N}\text{ such that }v_{y,x}(t)\in F_m(t,x)\cap\partial f(x)^\circ\text{ a.e. on }I\\&(v_{y,x}(t)\in F_m(t,x)\cap\partial(-f)(x)^\circ\text{ a.e. on }I\text{ respectively)}
\end{aligned}
\end{equation}
Fix $m\geqslant 1$ and $y\in\bd K$. Using the assumptions \eqref{liminf} we choose a measurable mapping $w_y\colon I\to\R{N}$ such that $w_y(t)\in F(t,y)\cap\Liminf_{x\to y,x\not\in K}\partial f(x)^\circ$ everywhere on a set $I_0$ of full measure in the interval $I$. Let $u_y\colon I\to\R{N}$ be a simple function such that $u_y(I)\subset w_y(I_0)$ and $u_y(t)\in B\left(w_y(t),\frac{r_m}{2}\right)$ on $I_0$. Clearly, there is $\delta_y\in(0,r_m)$, such that
\begin{equation}\label{u_y}
u_y(t)\in B\left(\partial f(x)^\circ,\frac{r_m}{2}\right)\;\;\;\text{for any }(t,x)\in I\times(B(y,\delta_y)\setminus K).
\end{equation}
Since $\bd K$ is compact, there are points $y_1,\ldots,y_k$ such that $\bd K\subset\bigcup_{i=1}^kB\left(y_i,\frac{\delta_{y_i}}{2}\right)$. Put $\delta_m:=\min\left\{\frac{\delta_{y_i}}{2}\colon i=1,\dots,k\right\}$. Let $y_{i_0}\in\bd K$ be such that $y\in B\left(y_{i_0},\frac{\delta_{y_{i_0}}}{2}\right)$. Take $x\in B(y,\delta_m)\setminus K$. Then $x\in B\left(y_{i_0},\delta_{y_{i_0}}\right)\setminus K$ and, by \eqref{u_y}, for all $t\in I$ \[B\left(u_{y_{i_0}}(t),\frac{r_m}{2}\right)\cap\partial f(x)^\circ\neq\varnothing.\] Evidently, there exists a measurable $u_{y,x}\colon I\to\R{N}$ satisfying $u_{y,x}(t)\in B\left(u_{y_{i_0}}(t),\frac{r_m}{2}\right)\cap\partial f(x)^\circ$ a.e. on $I$. Recall that if $x\in\supp\lambdaup_z^m$, then $|x-z|<r_m$. Hence, \[u_{y,x}(t)\in B\left(w_{y_{i_0}},r_m\right)\subset B\left(\w{F}(t,y_{i_0}),r_m\right)\subset B\left(\w{F}\left(t,B\left(y,\frac{r_m}{2}\right)\right),r_m\right)\subset\overline{\co}B\left(\w{F}(t,B(z,2r_m)),\frac{1}{m}\right)\] a.e. on $I$. Observe that \[\sum_{z\in\R{N}}\lambdaup_z^m(x)u_{y,x}(t)\in F_m(t,x)\cap\partial f(x)^\circ\] a.e. on $I$. Taking as $v_{y,x}$ the combination $\sum_{z\in\R{N}}\lambdaup_z^m(x)u_{y,x}$, we justify \eqref{v_x,y}.\par Denote by $K_n$ the strongly regular compact sublevel set $K_\frac{1}{n}=\left\{x\in\overline{\Omega}\colon f_\frac{1}{n}(x)\<0\right\}$, introduced in the discussion preceding the formulation of present theorem. Let $m\geqslant 1$ be arbitrary. Now choose $\delta_m>0$ as in claim \eqref{v_x,y}. Since $K=\bigcap_{n=1}^\infty K_n$, there is $n_m\geqslant m$ such that $K_{n_m}\subset B\left(K,\delta_m\right)$. Put $F_{n_m}:=\res{F_m}{I\times K_{n_m}}$. So defined multimap $F_{n_m}$ satisfies assumptions of Theorem \ref{th1}. In particular, one of the tangency conditions \eqref{f}-\eqref{-f} is met. Indeed, if $x\in\bd K_{n_m}$, then there is $y\in\bd K$ such that $x\in B(y,\delta_m)\setminus K$. Thus, in view of \eqref{v_x,y} either the multimap $F_{n_m}(\cdot,x)\cap\partial f_{\frac{1}{n_m}}(x)^\circ$ or $F_{n_m}(\cdot,x)\cap\partial(-f_{\frac{1}{n_m}})(x)^\circ$ possesses a measurable selection.\par In order to apply the thesis of Theorem \ref{th1}. we must verify assumptions \eqref{warprop2} and \eqref{boundarycond}. In our case the boundary operator $g$ satisfies $g(x,y)=x+y$. Take $x\in C^1\left(I,K_{n_m}\right)$ such that $g(x(0),x(T))=0$. Then
\begin{align*}
x(0)\in\bd K_{n_m}&\Leftrightarrow x(0)\in\dom(f)\setminus\Int K_{n_m}\stackrel{\text{\scriptsize{by \eqref{epilevel}}}}{\Longleftrightarrow}-x(0)\in\dom(f)\setminus\Int K_{n_m}\\&\Leftrightarrow x(T)\in\dom(f)\setminus\Int K_{n_m}\Leftrightarrow x(T)\in\bd K_{n_m}.
\end{align*}
Therefore, condition \eqref{boundarycond} is satisfied on the assumption of symmetry with respect to the origin of superlevel sets $f^{-1}([\eps,+\infty))$.\par Due to the assumption $0\in K$, we see that $f(0)<\frac{1}{n_m}$, i.e. $0\not\in\bd K_{n_m}$. Since $g(x_0,x_0)=2x_0$, it is apparent that $0\not\in g(\cdot,\cdot)\left(\bd K_{n_m}\right)$. It is also clear that \[|g(x_0,x_0)-x_0|<|g(x_0,x_0)|+|x_0|\] for any $x_0\in\bd K_{n_m}$. This inequality is a simple reformulation of the Poincar\'e-Bohl condition \cite[Th.2.1.]{zabreiko}, which means that vector fields $g(\cdot,\cdot)$ and $id_{K_{n_m}}$ are joined by the linear homotopy, nonsingular on the boundary of $K_{n_m}$. The homotopy invariance of the Brouwer degree entails
\[\deg\left(g(\cdot,\cdot),\Int K_{n_m},0\right)=\deg\left(id_{K_{n_m}},\Int K_{n_m},0\right)=1\neq 0,\] as $0\in K$. At this point we can invoke Theorem \ref{th1}., stating that the following constrained antiperiodic problem 
\begin{equation}\label{K_nproblem}
\begin{aligned}
&x'\in F_{n_m}(t,x),\;\text{a.e. on }I,x\in K_{n_m}\\
&x(0)=-x(T),
\end{aligned}
\end{equation}
possesses a solution.\par Denote by $x_m$ a solution of the antiperiodic problem \eqref{K_nproblem}. Since \[|x_m(t)|\<|x_m(0)|+\int_0^t|\dot{x}_m(s)|\,ds\<\sup_{y\in\overline{\Omega}}|y|+T(c+1)<+\infty\] on $I$ and $|\dot{x}_m(t)|\<c+1$ a.e. on $I$, by passing to a subsequence if necessary, we can assume that $(x_m)_{m\geqslant 1}$ converges uniformly to some absolutely continuous $x\colon I\to\R{N}$. Concurrently, $\dot{x}_m\rightharpoonup\dot{x}$ in $L^1(I,\R{N})$ (compare \cite[Th.0.3.4.]{aubin}). From \eqref{F_m} it follows that $\dot{x}_m(t)\in\overline{\co}\,B\left(\w{F}(t,B(x_m(t),3r_m)),\frac{1}{m}\right)$ a.e. on $I$. Thus, in view of Theorem \ref{convergence}. $\dot{x}(t)\in\w{F}(t,x(t))$ for a.a. $t\in I$. Considering that $x_m(I)\subset K_{n_m}$, it is evident that $x(t)\in K$ for all $t\in I$. Therefore, $x$ is the required solution to problem \eqref{anti}. 
\end{proof}
\begin{corollary}
Let $K$, represented by $f$, be a compact and strictly regular set containing origin and satisfying condition \eqref{epilevel}. Assume that $F\colon I\times K\map\R{N}$ fulfills $(\F_1)$-$(\F_4)$ and that either \eqref{f} or \eqref{-f} holds. Then the antiperiodic problem \eqref{anti} possesses at least one solution. In particular, if $K$ is represented by its distance function, then the tangency condition \eqref{f} is equivalent to \eqref{fepi}, while condition \eqref{-f} is equivalent to  
\[\forall\,x\in\bd K\;\;\;F(\cdot,x)\cap -C_K(x)\text{ has a measurable selection}.\]
\end{corollary}
\begin{remark}
Observe that if $\bd K$ is symmetric, then the distance function $\res{d_K}{K^c}$ is even. Therefore, if $K$ is a constraint set represented by its distance function $d_K$, then assumption \eqref{epilevel} is fulfilled, in particular, when the boundary $\bd K$ of this set is symmetric.  
\end{remark}
\begin{remark}
Unfortunately, the tangency conditions \eqref{liminf}-\eqref{-liminf} are rather restrictive, which can make problematic the demonstration of the existence of antiperiodic solutions to ODE's with topologically ``thin'' set of state constraints. In other words, Theorem \ref{ant}. is not fully satisfactory, what the following common example makes us aware of $($comp. \cite[p.203]{aubin}$)\!:$ Let $K:=S_1\cup S_{-1}$, where
\[S_i:=\left\{z=(x,y)\in\R{2}\colon (x-i)^2+y^2=1\right\}.\] Put
\[g(x,y):=
\begin{cases}
(y,1-x)&\text{for }(x,y)\in S_1,\\
(-y,1+x)&\text{for }(x,y)\in S_{-1}.
\end{cases}\]
Then, for all $z\in K$, $g(z)\in T_K(z)=C_K(z)$. Moreover, for every $z\in K\setminus\{(0,0)\}$ there is $\delta>0$ such that $\res{d_K}{B(z,\delta)}=\res{d_{S_i}}{B(z,\delta)}$. Taking into account that $S_i$ is also strictly regular, we see that \[g(z)\in C_K(z)=\partial d_K(z)^\circ=\partial d_{S_i}(z)^\circ\subset\underset{u\to z,u\notin S_i}{\Liminf}\,\partial d_{S_i}(u)^\circ=\underset{u\to z,u\notin K}{\Liminf}\,\partial d_K(u)^\circ.\] Actually, \[g(z)=g(-z)\in\underset{u\to -z,u\notin K}{\Liminf}\,\partial d_K(u)^\circ=\underset{u\to -z,u\notin K}{\Liminf}\,\partial(-d_K)(-u)^\circ=\underset{u\to z,u\notin K}{\Liminf}\,\partial(-d_K)(u)^\circ,\] for $g$ and $d_K$ are even. However, on the other hand \[g(0,0)\not\in\underset{u\to(0,0),u\notin K}{\Liminf}\,\partial d_K(u)^\circ=\underset{u\to(0,0),u\notin K}{\Liminf}\,\partial(-d_K)(u)^\circ.\] Therefore, Theorem \ref{ant}. does not confirm the obvious observation that the autonomous ODE $\dot{x}(t)=g(x(t))$, for $t\in I$, possesses antiperiodic solutions $($there are exactly two, starting from each point $z\in K)$.
\end{remark}

In a situation where the set of state constraints possesses extremely simplified geometry, i.e. it is disc shaped, then the tangency conditions can be formulated in such a way as it was done in the following theorem.
\begin{theorem}\label{th2}
Suppose $F\colon I\times\R{N}\map\R{N}$ satisfies conditions $(\F_1)$-$(\F_4)$, with $K:=\R{N}$. Moreover, assume that there exists $r>0$ such that either
\begin{equation}\label{F_5'}
\forall\,x\in\bd D(0,r)\;\;\;F(\cdot,x)\cap \{x\}^\circ\text{ possesses a measurable selection}
\end{equation}
or
\begin{equation}\label{-F_5'}
\forall\,x\in\bd D(0,r)\;\;\;F(\cdot,x)\cap \{-x\}^\circ\text{ possesses a measurable selection.}
\end{equation}
Let $g\colon C\left(I,\R{N}\right)\to\R{N}$ be a continuous function, which maps bounded sets into bounded sets and there exists $\eps_0>0$ such that the following conditions hold:
\begin{itemize}
\item[(i)] $\forall\,x\in C^1(I,B(0,r+\eps_0))\;\;\;[x(0)=g(x)]\Rightarrow\exists\,t\in(0,T]\;\;|g(x)|\<|x(t)|$,
\item[(ii)] $|g(i(x))|\<|x|\;$ for all $x\in B(0,r+\eps_0)\setminus D(0,r)$,
\item[(iii)] $\fix(g\circ i)\cap(B(0,r+\eps_0)\setminus D(0,r))=\varnothing$.
\end{itemize}
Then the nonlocal Cauchy problem 
\begin{equation}\label{uncons}
\begin{gathered}
x'\in F(t,x),\;\text{a.e. on }I,\\
x(0)=g(x)
\end{gathered}
\end{equation}
possesses at least one solution $x$ such that $x(t)\in D(0,r)$ for all $t\in I$.
\end{theorem}
\begin{proof}
The disc $D(0,r)$ represented by its distance function $d_{D(0,r)}\colon\R{N}\to\R{}$ is strictly regular. Unfortunately, it is not strongly regular within the meaning of Definition \ref{K}. However, the sublevel set $D(0,r+\eps)=\left\{x\in\R{N}\colon d_{D(0,r)}(x)\<\eps\right\}$ is compact  and strongly regular, for every $\eps\in(0,\eps_0)$. The closed half-space $\{x\}^\circ=\{y\in\R{N}\colon\langle x,y\rangle\<0\}$ is nothing more than the cone $S_{\!D(0,r)}(x):=\overline{\bigcup_{h>0}\frac{1}{h}D(-x,r)}$ tangent to the convex subset $D(0,r)$ at $x$. The latter stands for the common value of the Bouligand contingent and Clarke tangent cones, i.e. $S_{\!D(0,r)}(x)=T_{\!D(0,r)}(x)=C_{D(0,r)}(x)$ (see \cite[Prop.4.2.1]{fran}). On the other hand it is well known that $C_K(x)=\partial d_K(x)^\circ$. Thus, the tangency condition \eqref{F_5'} is equivalent to
\begin{equation}\label{tangcon1}
\forall\,x\in\bd D(0,r)\;\;\;F(\cdot,x)\cap\partial d_{D(0,r)}(x)^\circ\text{ has a measurable selection.}
\end{equation}
Bearing in mind that $\{-x\}^\circ=-\{x\}^\circ=-\,\partial d_{D(0,r)}(x)^\circ=\partial(-d_{D(0,r)})(x)^\circ$ there is also an equivalence between the tangency condition \eqref{-F_5'} and 
\begin{equation}\label{tangcon2}
\forall\,x\in\bd D(0,r)\;\;\;F(\cdot,x)\cap\partial(-d_{D(0,r)})(x)^\circ\text{ has a measurable selection.}
\end{equation}
Referring to an argument contained in the proof of \cite[Cor.1.12]{bader} we find that \eqref{tangcon1} entails
\[\forall\,x\in\bd D(0,r)\;\;\;F(\cdot,x)\cap\underset{y\to x,\,y\notin D(0,r)}{\Liminf}\,\partial d_{D(0,r)}(y)^\circ\text{ has a measurable selection,}\] while \eqref{tangcon2} implies \[\forall\,x\in\bd D(0,r)\;\;\;F(\cdot,x)\cap\underset{y\to x,\,y\notin D(0,r)}{\Liminf}\,\partial(-d_{D(0,r)})(y)^\circ\text{ has a measurable selection.}\]
\par Let $F_m\colon I\times\R{N}\map\R{N}$ be given by the formula \eqref{Fmform}. Retaining the notation and the course of reasoning of the proof of Theorem \ref{ant}. it is easy to convince oneself that the nonlocal Cauchy problem 
\begin{equation}\label{ring}
\begin{aligned}
&x'\in F_{n_m}(t,x),\;\text{a.e. on }I,x\in D\left(0,r+\frac{1}{n_m}\right)\\
&x(0)=g(x),
\end{aligned}
\end{equation}
possesses at least one solution (for $n_m$ large enough). A closer look at Theorem \ref{maw}. and the proof of Theorem \ref{th1}. indicates that
it is sufficient to verify conditions 
\[\deg\left(id_{D_m}-g\circ\res{i}{D_m},\Int D_m,0\right)\neq 0\] and \[\ev_0\left(\left\{x\in C^1\left(I,D_m\right)\colon x((0,T])\subset\Int D_m\right\}\cap(ev_0-g)^{-1}(0)\right)\subset\Int D_m,\]
where $D_m:=D\left(0,r+\frac{1}{n_m}\right)$. Assumptions (ii)-(iii), related to the boundary operator $g$, imply \[\left|x_0-\left(g\circ\res{i}{D_m}\right)-x_0\right|<\left|\left(g\circ\res{i}{D_m}\right)-x_0\right|+|x_0|\] for every $x_0\in\bd D_m$. Therefore, the linear homotopy joining the identity $id_{D_m}$ with the vector field $id_{D_m}-g\circ\res{i}{D_m}$ is nonsingular on the sphere $\bd D_m$. As a result, \[\deg\left(id_{D_m}-g\circ\res{i}{D_m},\Int D_m,0\right)=\deg\left(id_{D_m},\Int D_m,0\right)=1\neq 0.\] Now, take $x\in C^1\left(I,D_m\right)$ such that $x((0,T])\subset\Int D_m$ and $x(0)=g(x)$. By virtue of assumption (i), there exists $t\in(0,T]$ such that $|g(x)|\<|x(t)|$. However, since $x((0,T])\subset\Int D_m$, we see that $|x(t)|<r+\frac{1}{n_m}$. Thus, $|x(0)|<r+\frac{1}{n_m}$ and eventually $x(0)\in\Int D_m$.\par If $x_m$ stands for the solution of \eqref{ring}, then there exists a uniformly convergent subsequence of $(x_m)_{m\geqslant 1}$. The limit of this subsequence constitutes the solution of problem \eqref{uncons}, for exactly the same reasons as in the proof of the preceding theorem. 
\end{proof}
\begin{remark}
Analysis of the proof of Theorem \ref{th1}. enables a more complete interpretation of the tangency condition \eqref{F_5'}. Let us point out that the distance function $d_{D(0,r)}$ plays a role of a generalized convex coercive guiding potential for the approximating function $W_n^k$ outside the ball $D(0,r)$ $($compare \cite{blasi,lewicka}$)$.
\end{remark}
\begin{corollary}
Let $0<t_1<t_2<\ldots<t_n\<T$ be arbitrary, but fixed. Assume that $\sum_{i=1}^n|\alpha_i|\<1$ and $\sum_{i=1}^n\alpha_i\neq 1$. Under the assumptions $(\F_1)$-$(\F_4)$ and \eqref{F_5'} or \eqref{-F_5'} the boundary value problem with multi-point discrete mean condition 
\begin{equation*}
\begin{gathered}
x'\in F(t,x),\;\text{a.e. on }I,\\
x(0)=\sum\limits_{i=1}^n\alpha_ix(t_i)
\end{gathered}
\end{equation*}
possesses a solution. In particular, the respective antiperiodic problem has at least one solution.
\end{corollary}
\begin{proof}
It suffices to realize that the boundary operator corresponding to nonlocal initial condition \eqref{multi2} satisfies assumptions (i)-(iii) of Theorem \ref{th2}. This is the case because \[|x(0)|=\left|\sum\limits_{i=1}^n\alpha_ix(t_i)\right|\<\sum\limits_{i=1}^n|\alpha_i||x(t_i)|\<\max_{1\<i\<n}|x(t_i)|=|x(t_{i_0})|\] for some $t_{i_0}\in(0,T]$.
\end{proof}
\begin{corollary}
Let $h\colon\R{N}\to\R{N}$ be a continuous map such that $|h(x)|\<|x|$ for any $x\in\R{N}$. Assume further that the fixed point set of $h$ is compact. Under the assumptions $(\F_1)$-$(\F_4)$ and \eqref{F_5'} or \eqref{-F_5'} the boundary value problem with mean value condition 
\begin{equation*}
\begin{gathered}
x'\in F(t,x),\;\text{a.e. on }I,\\
x(0)=\frac{1}{T}\int_0^Th(x(t))\,dt
\end{gathered}
\end{equation*}
possesses a solution.
\end{corollary}
\begin{proof}
Let $r>0$ be such that $\fix(h)\subset D(0,r)$. Put $g(x):=(1/T)\int_0^Th(x(t))\,dt$. Observe that $\fix(h)\cap(\R{N}\setminus D(0,r))=\varnothing$ and $\fix(g\circ i)=\fix(h)$. Now, suppose that $x(0)=g(x)$ and $|x(0)|>|x(t)|$ for $t\in(0,T]$. Then
\[|x(0)|=\frac{1}{T}\int_0^T|x(0)|\,dt>\frac{1}{T}\int_0^T|x(t)|\,dt\geqslant\frac{1}{T}\int_0^T|h(x(t))|\,dt\geqslant\left|\frac{1}{T}\int_0^Th(x(t))\,dt\right|=|x(0)|.\] Thus, there exists $t_0\in(0,T]$ such that $|x(0)|\<|x(t_0)|$.
\end{proof}
Employing an extension of the concept introduced in Definition \ref{boundset}. we will now investigate an issue of the existence of solutions to Floquet boundary value problem associated with differential inclusion. Assume that $C\in GL(N,\R{})$. Denote by $\langle C\rangle$ the cyclic subgroup generated by $C$. Let $P_N\colon\R{N}\to\R{N}$ be a linear projector whose range is $\ker(id-C)$. For $F\colon I\times\R{N}\map\R{N}$ satisfying $(\F_1)$-$(\F_4)$ one can define, engaging the notion of the Aumann integral, a set-valued map $\w{F}\colon\R{N}\map\R{N}$ by the formula $\w{F}(x_0):=\int_0^TF(t,x_0)\,dt$.
\begin{theorem}\label{floqth}
Let $K$ be an open bounded subset of $\R{N}$ such that $\bd K$ is inavariant under the action of $\langle C\rangle$. Let $F\colon I\times\overline{K}\map\R{N}$ be a bounded upper semicontinuous map with convex compact values. Assume that, for each $x\in\bd K$, there exists a locally Lipschitz and regular at $x$ function $f_x\colon\dom(f_x)\to\R{}$ such that the following conditions hold:
\begin{itemize}
\item[(i)] $\overline{K}\subset\{y\in\dom(f_x)\colon f_x(y)\<0\}$,
\item[(ii)] $f_x(x)=0$,
\item[(iii)] $\forall\,t\in I\;\forall\,y\in F(t,x)\;\;f^\circ_x(x;y)\neq 0$,
\item[(iv)] $\max\limits_{y\in F(0,x)}f_x^\circ(x;y)\cdot\max\limits_{z\in F(T,Cx)}f_{Cx}^\circ(Cx;-z)<0$.
\end{itemize}
Moreover, assume that either
\begin{itemize}
\item[(v)] $\ker(id-C)=\{0\}$ and $0\in K$,
\end{itemize}
or
\begin{itemize}
\item[(vi)] $\int\limits_0^TF(t,x_0)\,dt\cap\im(id-C)=\varnothing$ for $x_0\in\ker(id-C)\cap\bd K$ and \[\deg(P_N\circ\w{F},\ker(id-C)\cap K,0)\neq 0.\]
\end{itemize}
Then the following Floquet boundary value problem 
\begin{equation}\label{floquet}
\begin{gathered}
x'\in F(t,x),\;t\in I,x\in\overline{K}\\
x(T)=Cx(0),
\end{gathered}
\end{equation}
has at least one solution.
\end{theorem}
\begin{proof}
We will adjust the Fredholm setting depending on whether we are dealing with the case (v) or (vi).\par Ad (v): Let $X:=C(I,\R{N})$, $Z:=L^1(I,\R{N})\times\R{N}$, $\dom L:=AC(I,\R{N})$; $L\colon\dom L\subset X\to Z$ be such that $Lx:=(\dot{x},0)$. Then $\ker L=i(\R{N})$, where $i\colon\R{N}\hookrightarrow C(I,\R{N})$ is defined by $i(x_0)(t)=x_0$, $\im L=L^1(I,\R{N})\times \{0\}$ and $\coker L\approx\R{N}$, i.e. $L$ is a Fredholm mapping of index zero. Consider continuous linear operators $P\colon X\to X$ and $Q\colon Z\to Z$ such that $P:=\ev_0$ and $Q((y,x_0)):=(0,x_0)$. It is clear that $(P,Q)$ is an exact pair of idempotent projections with respect to $L$. From now on $\Phi\colon\im Q\to\ker L$ will denote a fixed linear homeomorphism, given by $\Phi((0,x_0))=i(x_0)$. Put $\Omega:=C(I,K)$ and define a nonlinear operator $N\colon\overline{\Omega}\to Z$ by \[N(x):=\left\{w\in L^1(I,\R{N})\colon w(t)\in F(t,x(t))\mbox{ a.e. on }I\right\}\times\left\{\gamma(x)\right\},\] where $\gamma:=C\circ\ev_0-\ev_T$.  It is clear that the Floquet boundary value problem \eqref{floquet} is equivalent to the operator inclusion $Lx\in N(x)$. \par It is a matter of routine to verify that maps $QN$ and $K_{P,Q}N$ satisfy the assumptions of Theorem \ref{conth}. In order to confirm that the condition (a) of Theorem \ref{conth}. is also fulfilled let us suppose that $x\in\overline{\Omega}=C(I,\overline{K})$ is a solution to the boundary value problem
\begin{equation}\label{floq2}
\begin{aligned}
&\dot{x}(t)\in\lambdaup F(t,x(t)),\;\text{a.e. on }I,\\
&x(T)=Cx(0),
\end{aligned}
\end{equation}
for some $\lambdaup\in(0,1)$. In view of the invariance of the boundary $\bd K$ with respect to the cyclic subgroup $\langle C\rangle$ we see that $x(0)\in\bd K\Leftrightarrow x(T)\in\bd K$. Suppose for definiteness that $x(0)\in\bd K$. As we know from \cite[Def.2.6.1]{clarke} and \cite[Th.4.]{warga} the generalized Jacobian of $x$ at $t\in(0,T)$ is defined by \[\partial x(t):=\co\left\{\lim_{n\to\infty}\dot{x}(t_n)\colon t_n\xrightarrow[n\to\infty]{}t, t_n\not\in(\Omega_x\cup S)\right\},\] where $S\subset(0,T)$ is a set of Lebesgue measure zero and $\Omega_x$ is a set of points, where $x$ is not differentiable. By analogy we may define one-sided generalized Jacobians of $x$ at boundary points, i.e. \[\partial_+x(0):=\overline{\co}\left\{\lim_{n\to\infty}\dot{x}(t_n)\colon t_n\xrightarrow[n\to\infty]{}0^+, t_n\not\in(\Omega_x\cup S)\right\}\] and \[\partial x_-(T):=\overline{\co}\left\{\lim_{n\to\infty}\dot{x}(t_n)\colon t_n\xrightarrow[n\to\infty]{}T^-, t_n\not\in(\Omega_x\cup S)\right\}.\] The following observations concerning the introduced notions are straightforward:
\begin{itemize}
\item[(a)] $\partial x$ is upper semicontinuous at $0$ (respectively, at $T$) in the sense that \\$\forall\,\eps>0\;\exists\,\delta>0\;\forall\,t\in(0,\delta]\;\;\partial x(t)\subset B\left(\partial_+x(0),\eps\right)$\\$\left(\forall\,\eps>0\;\exists\,\delta>0\;\forall\,t\in[T-\delta,T)\;\;\partial x(t)\subset B\left(\partial_-x(T),\eps\right)\right)$ - it is a plain consequence of \cite[Prop.2.6.2(c)]{clarke},
\item[(b)] $\partial_+x(0)\subset F(0,x(0))$ and $\partial_-x(T)\subset F(T,x(T))$ by analogy to \cite[Lem.3.1.]{blasi},
\item[(c)] $x(h)-x(0)\in\overline{\co}\left\{\partial x((0,h])\cup\partial_+x(0)\right\}\cdot h$ and \\$x(T)-x(T-h)\in\overline{\co}\left\{\partial x([T-h,T))\cup\partial_-x(T)\right\}\cdot h$ for $h>0$ - along the lines of arguments in \cite[Prop.2.6.5]{clarke}.
\end{itemize}
According to the assumption, there is a locally Lipschitz $f=f_{x(0)}$ such that $f(x(0))=0$ and $f(x(h))\<f(x(0))$ for $h>0$. In view of (c) we may describe an increment of the function $x$ as follows $x(h)-x(0)=p_h\cdot h$ for some $p_h\in\overline{\co}\left\{\partial x((0,h])\cup\partial_+x(0)\right\}$. Taking into account (a) it becomes clear that $p_h\xrightarrow[h\to 0^+]{}p_0\in\partial_+x(0)$. Put $\sigma(h):=(p_h-p_0)h$. Now, we are in position to estimate
\begin{equation}\label{est1}
\begin{split}
0&\geqslant\underset{h\to 0^+}{\overline{\lim}}\;\frac{f(x(h))-f(x(0))}{h}\\&=\underset{h\to 0^+}{\overline{\lim}}\!\left(\frac{f(x(h))-f(x(0)+\sigma(h))}{h}+\frac{f(x(0)+\sigma(h))-f(x(0))}{h}\right)\\&=\underset{h\to 0^+}{\overline{\lim}}\!\frac{f(x(h))\!-\!f(x(0)+\sigma(h))}{h}+\!\underset{h\to 0^+}{\lim}\!\frac{f(x(0)+\sigma(h))\!-\!f(x(0))}{h}\\&=\underset{h\to 0^+}{\overline{\lim}}\;\frac{f(x(0)+\sigma(h)+hp_0)-f(x(0)+\sigma(h))}{h}+L_{x(0)}\underset{h\to 0^+}{\lim}\;\frac{|\sigma(h)|}{h}\\&=\underset{h\to 0^+}{\lim}\;\frac{f(x(0)+hp_0)-f(x(0))}{h}=f^\circ(x(0);p_0),
\end{split}
\end{equation}
where $L_{x(0)}$ is a Lipschitz constant for function $f$ in the neighbourhood of the point $x(0)$. The last equality is a consequence of regularity of $f$ at $x(0)$. 
\par Observe that $x(T-h)-x(T)=p_h\cdot(-h)$, where $p_h\in\overline{\co}\left\{\partial x([T-h,T))\cup\partial_-x(T)\right\}$. Moreover, $p_h\xrightarrow[h\to 0^+]{}p_T\in\partial_-x(T)$. Let $f:=f_{x(T)}$ and $\sigma(-h):=(p_h-p_T)(-h)$. A quite similar reasoning to the previous one leads to the conclusion that
{\allowdisplaybreaks\begin{equation}\label{est2}
\begin{split}
0&\<\underset{h\to 0^+}{\overline{\lim}}\;\frac{f(x(T-h))-f(x(T))}{-h}\\&=\underset{h\to 0^+}{\overline{\lim}}\!\left(\frac{f(x(T-h))-f(x(T)+\sigma(-h))}{-h}+\frac{f(x(T)+\sigma(-h))-f(x(T))}{-h}\right)\\&=\underset{h\to 0^+}{\overline{\lim}}\!\frac{f(x(T-h))\!-\!f(x(T)+\sigma(-h))}{-h}+\!\underset{h\to 0^+}{\lim}\!\frac{f(x(T)+\sigma(-h))\!-\!f(x(T))}{-h}\\&=\underset{h\to 0^+}{\overline{\lim}}\;\frac{f(x(T)+\sigma(-h)-hp_T)-f(x(T)+\sigma(-h))}{-h}-L_{x(T)}\underset{h\to 0^+}{\lim}\;\frac{|\sigma(-h)|}{|-h|}\\&\<\underset{h\to 0^+}{\underset{y\to x(T)}{\overline{\lim}}}\frac{f(y-hp_T)-f(y)}{-h}=-f^\circ(x(T);-p_T).
\end{split}
\end{equation}}
In summary, estimates \eqref{est1}-\eqref{est2} come down to the following conclusion
\[\exists\,p_0\in F(0,x(0))\;\exists\,p_T\in F(T,Cx(0))\;\;\;f^\circ_{x(0)}(x(0);p_0)\cdot f^\circ_{Cx(0)}(Cx(0);-p_T)\geqslant 0,\]
which yields a contradiction with the assumption (iv). Thus, $x(\{0,T\})\cap\bd K=\varnothing$.\par Seeing that $F$ has connected values and the derivative $f^\circ(x;\cdot)$ is continuous, one may formulate equivalently the assumption (iii) in the following way: for every $t\in I$ the expression $f_x^\circ(x;y)$ is of constant sign independently of the choice of an argument $y\in F(t,x)$. Suppose there is a point $t_0\in(0,T)$ such that $x(t_0)\in\bd K$. Let $f=f_{x(t_0)}$ satisfy conditions (i)-(iii). Then $f(x(t))\<f(x(t_0))$ on some both-sided neighbourhood of $t_0$. Imitating the line of reasoning accompanying estimate \eqref{est1} we find that there is $p_0\in\partial x(t_0)\subset F(t_0,x(t_0))$ such that
\[f^\circ(x(t_0);p_0)=\underset{h\to 0^+}{\overline{\lim}}\;\frac{f(x(t_0+h))-f(x(t_0))}{h}\<0.\] This inequality is contrary to the assumption that $f^\circ(x(t_0);y)>0$ for all $y\in F(t_0,x(t_0))$. On the other hand, estimating in a way presented in \eqref{est2} one sees that \[f^\circ(x(t_0);-p_0)=-\underset{h\to 0^+}{\overline{\lim}}\;\frac{f(x(t_0-h))-f(x(t_0))}{-h}\<0,\] for some $p_0\in F(t_0,x(t_0))$. It means that there exists $p\in\partial f(x(t_0))$ such that $\langle p,p_0\rangle\geqslant 0$. The requirement that $f^\circ(x(t_0);y)<0$ for every $y\in F(t_0,x(t_0))$ is obviously contrary to the latter. Therefore, $x(I)\cap\bd K=\varnothing$. In fact, we have confirmed that $Lx\not\in\lambdaup N(x)$ for all $x\in\dom L\cap\bd\Omega$ and $\lambdaup\in(0,1)$.\par It is clear that condition $0\not\in QN(\ker L\cap\bd\Omega)$ is equivalent to $(id-C)x_0\neq 0$ for $x_0\in\bd K$. The latter is obviously satisfied, due to the assumption (v). Using standard properties of the Brouwer degree we obtain
\begin{align*}
\deg(\Phi QN,\ker L\cap\Omega,0)&=\deg(i^{-1}\Phi QNi,i^{-1}(\Omega),0)=\deg(id-C,K,0)\\&=\sgn\det(id-C)=\pm 1\neq 0,
\end{align*}
as $0\in(id-C)(K)$ and $(id-C)\in GL(N,\R{})$. In wiev of Theorem \ref{conth}. there is a solution $x\in\overline{\Omega}$ of the inclusion $Lx\in N(x)$. This is of course also a solution of the problem \eqref{floquet}. \par Ad (vi): Put $X:=C(I,\R{N})$ and $Z:=L^1(I,\R{N})$. Let us redefine the domain $\dom L:=\left\{x\in AC\left(I,\R{N}\right)\colon x(T)=Cx(0)\right\}$ of operator $L$ given by $Lx:=\dot{x}$. In this case $\ker L=i(\ker(id-C))$ and \[\im L=\left\{x\in L^1(I,\R{N})\colon\int_0^Tx(s)\,ds\in\im(id-C)\right\}.\] It is easy to see that $\coker L\approx\ker(id-C)$. Therefore, $L$ is a Fredholm mapping of index zero. Consider a continuous linear operator $P\colon X\to X$ such that $P:=i\circ P_N\circ\ev_T$. It is clear that $\im P=\ker L$. Define the second projection $Q\colon Z\to Z$ by the formula $Q:=i\circ P_N\circ\ev_T\circ V$, where $V\colon Z\to X$ is an integral operator given by $V(x)(t):=\int_0^tx(s)\,ds$. It follows directly that $\ker Q=\im L$. Observe that $\ev_T(X)=\R{N}=\ev_T(V(Z))$, so we are allowed to choose $\Phi:=id_{\,\im Q}$. Take $z\in\im L$. There is a unique $x_0\in\R{N}$ such that $\int_0^Tz(s)\,ds=(id-C)(id-P_N)x_0$. Therefore, the inverse $L_P^{-1}\colon\im L\to\dom L\cap\ker P$ should be defined as follows $L^{-1}_P(z):=(id-P_N)x_0+V(z)$. Since the restriction $L_P$ is a bounded linear homomorphism between the Banach spaces $\left(\dom L\cap\ker P,||\cdot||_{AC}\right)$ and $\left(\im L,||\cdot||_1\right)$, the inverse is also bounded. Again, denote by $\Omega$ the open subset $C(I,K)$. Let $N\colon\overline{\Omega}\to Z$ be the Nemytskij operator corresponding to $F$, i.e. \[N(x):=\left\{w\in L^1(I,\R{N})\colon w(t)\in F(t,x(t))\mbox{ a.e. on }I\right\}.\] Once more we may rewrite the Floquet boundary value problem \eqref{floquet} in the form of the operator inclusion $Lx\in N(x)$. \par For an arbitrary element $w\in N(\overline{\Omega})$ we have estimations
\[||K_{P,Q}(w)||=||L^{-1}_P(id-Q)(w)||\<||L_P^{-1}||_{\mathcal L}||id-Q||_{\mathcal L}Tc\] and 
\begin{align*}
\left|K_{P,Q}(w)(t)-K_{P,Q}(w)(\tau)\right|&=\left|L^{-1}_P(id-Q)(w)(t)-L^{-1}_P(id-Q)(w)(\tau)\right|\\&=\left|\int_0^t(id-Q)(w)(s)\,ds-\int_0^\tau(id-Q)(w)(s)\,ds\right|\\&\<\int_\tau^t|w(s)|\,ds+\left|Q\left(\int_\tau^tw(s)\,ds\right)\right|\\&\<c|t-\tau|+||Q||_{\mathcal L}c|t-\tau|,
\end{align*}
where the constant $c>0$ is such that for all $x\in\overline{K}$, $\sup_{y\in F(t,x)}|y|\<c$ a.e. on $I$. Hence, the set $K_{P,Q}N(\overline{\Omega})$ is relatively compact in the topology of uniform convergence. Analogously, the image $V(N(\overline{\Omega}))$ is relatively compact and eventually also image $QN(\overline{\Omega})$. The verification of upper semicontinuity of $QN$ and $K_{P,Q}N$ is completely standard and is based in part on the compactness of these maps.\par Justification of the condition (a) of Theorem \ref{conth}. is fully analogous to the arguments put forward previously. Notice that $QN(i(x_0))=(i\circ P_N\circ\w{F})(x_0)$. Thus, assumption $0\not\in QN(\ker L\cap\bd\Omega)$ is equivalent to \[\w{F}(\ker(id-C)\cap\bd K)\cap\ker P_N=\varnothing.\] In other words, it is equivalent to $\int\limits_0^TF(t,x_0)\,dt\cap\im(id-C)=\varnothing$ for $x_0\in\ker(id-C)\cap\bd K$. Now it suffices to note that \[\deg(\Phi QN,\ker L\cap\Omega,0)=\deg(P_N\circ\w{F},\ker(id-C)\cap K,0)\neq 0\] to be able to apply Theorem \ref{conth}. and get a solution of problem \eqref{floquet}.
\end{proof}
\begin{remark}
In fact, the application of the property of $K$ of being a subset of the sublevel set of function $f_x$ has merely a local nature. Therefore, assumption {\em (i)} of Theorem \ref{floqth}. can be reformulated as follows: there exists $\delta>0$ such that $\overline{K}\cap B(x,\delta)\subset\{y\in\dom(f_x)\colon f_x(y)\<0\}$.
\end{remark}

Let $\nabla_y^+f(x)$ denote the upper Dini directional derivative of the function $f$ at $x$ in the direction $y$, i.e.
\[\nabla_y^+f(x):=\underset{h\to 0^+}{\limsup}\;\frac{f(x+hy)-f(x)}{h}.\] The symbol $\langle A,B\rangle^{\pm}$ we use below stands for the lower (upper) inner product of nonempty compact subsets of $\R{n}$, i.e. 
\[\langle A,B\rangle^{-}=\inf\,\{\langle a,b\rangle\colon a\in A,b\in B\},\;\;\;\langle A,B\rangle^{+}=\sup\,\{\langle a,b\rangle\colon a\in A,b\in B\}.\]
\begin{corollary}\label{wn3}
The statement of Theorem \ref{floqth}. remains valid if, for every $x\in\bd K$, there is a locally Lipschitz function $f_x\colon\dom(f_x)\to\R{}$, which satisfies 
\begin{itemize}
\item[(iii')] $\forall\,t\in I\;\;\;\langle\partial f_x(x),F(t,x)\rangle^->0\vee\langle\partial f_x(x),F(t,x)\rangle^+<0$,
\item[(iv')] $\max\limits_{y\in F(0,x)}\nabla_y^+f_x(x)\cdot\max\limits_{z\in F(T,Cx)}f_{Cx}^\circ(Cx;-z)<0$.
\end{itemize}
instead of assumptions {\em (iii)-(iv)}.
\end{corollary}
\begin{proof}
In fact, condition (iii') is equivalent to the alternative: \[\forall\,t\in I\;\;\langle\partial f_x(x),F(t,x)\rangle^->0\,\text{ or }\,\forall\,t\in I\;\;\langle\partial f_x(x),F(t,x)\rangle^+<0.\] One easily sees that the assumption $\langle\partial f_x(x),F(t,x)\rangle^+<0$ for every $t\in I$ remains in contradiction with estimation \eqref{est2}.\par Let $x$ be a solution to \eqref{floq2} such that $x(t_0)\in\bd K$ for $t_0\in(0,T)$.  In the context of estimation \eqref{est1}, we may apply Lebourg theorem \cite[Th.2.3.7]{clarke} and the upper semicontinuity of generalized gradient $\partial f_{x(t_0)}$ to indicate a point $p_*\in\partial f_{x(t_0)}(x(t_0))$ for which $\nabla_y^+f_{x(t_0)}(x(t_0))=\langle p_*,y\rangle$. Then the contradiction with condition $\forall\,t\in I\;\langle\partial f_x(x),F(t,x)\rangle^->0$ becomes clear.
\end{proof}
\begin{remark}
If the mappings $f_x$, corresponding to $x\in\bd K$, are of $C^1$-class and the right-hand side $F$ of \eqref{floquet} is univalent, then the set of assumptions {\em (i)-(iv)} of Theorem \ref{floqth}. reduces to those given in \cite[Def.2.1.]{mawhin}. Therefore, properties {\em (i)-(iv)} $($or {\em (i)-(ii)} and {\em (iii')-(iv')}$)$ of the set $K$ may be treated as a natural generalization of the concept of a bound set to the case of differential inclusions with Floquet boundary condition.
\end{remark}
In the following result we do not require that the set of state constraints possesses characteristics of a bound set for Floquet boundary value problem and the indispensable tangency conditions are expressed in terms of vectors being normal to the constraint set in the sense of Bouligand.
\begin{theorem}\label{floqth2}
Let $K$ be an open bounded subset of $\R{N}$ such that $\bd K$ is inavariant under the action of $\langle C\rangle$. Let $F\colon I\times\overline{K}\map\R{N}$ be a bounded upper semicontinuous map with convex compact values. Assume that, for each $x\in\bd K$, there exists a nonzero regular normal vector $v(x)\in N_K^b(x)$ such that the following conditions hold:
\begin{itemize}
\item[(i)] $\forall\,t\in I\;\;\;F(t,x)\cap\{v(x)\}^\bot=\varnothing$,
\item[(ii)] $\min\limits_{y\in F(0,x)}\langle v(x),y\rangle\cdot\min\limits_{z\in F(T,Cx)}\langle v(Cx),z\rangle>0$.
\end{itemize}
Moreover, assume that either assumption {\em (v)} or {\em (vi)} of Theorem \ref{floqth}. is satisfied. Then the Floquet boundary value problem \eqref{floquet} possesses a solution.
\end{theorem}
\begin{proof}
Fix $x\in\bd K$. Notice that the image $F(I\times\{x\})$ is connected. By reffering to assumption (i) and the Darboux property we see that for all $t\in I$ and $y\in F(t,x)$ the term $\langle v(x),y\rangle$ is of constant sign. Since the multimap $\langle\cdot,F(\cdot,x)\rangle\colon\R{N}\times I\map\R{}$ is usc, it follows that \[\forall\,t\in I\;\exists\,\eps_t>0\;\forall\,s\in[t-\eps_t,t+\eps_t]\;\forall\,z\in D(v(x),\eps_t)\;\forall y\in F(s,x)\;\;\langle z,y\rangle\gtrless 0.\] Making use of the compactness of interval $I$ we deduce that there exists $\eps_x>0$ such that \[\forall\,t\in I\;\forall y\in F(t,x)\;\forall\,z\in D(v(x),\eps_x)\;\;\langle z,y\rangle\gtrless 0.\] In an analogous manner, we can match the number $\eps_{Cx}>0$, corresponding with the choice of the point $Cx\in\bd K$. On the other hand, in conjunction with the assumption (ii) and upper semicontinuity of the set-valued map $\langle\cdot,F(0,x)\rangle\cdot\langle\cdot,F(T,Cx)\rangle\colon\R{N}\times\R{N}\map\R{}$, there exists $\eps'>0$ such that \[\forall\,u\in D(v(x),\eps')\;\forall\,w\in D(v(Cx),\eps')\;\forall\,y\in F(0,x)\;\forall\,z\in F(T,Cx)\;\;\;\langle u,y\rangle\cdot\langle w,z\rangle>0.\]\par Put $\eps:=\min\{\eps_x,\eps_{Cx},\eps'\}$. In view of \cite[Prop.4.4.1]{fran} there exists $\delta>0$ such that \[\forall\,y\in\overline{K}\cap B(x,\delta)\;\;\;\langle v(x),y-x\rangle\<\eps|y-x|.\] Let $\dom(f_x):=B(x,\delta)$ and $f_x\colon\dom(f_x)\to\R{}$ be a Lipschitz function given by the formula $f_x(y):=\langle v(x),y-x\rangle-\eps|y-x|$. It is easy to calculate that $\partial f_x(x)=D(v(x),\eps)$. Furthermore, so defined mapping possesses also the following properties:
\begin{itemize}
\item[(i)] $\overline{K}\cap B(x,\delta)\subset\{y\in\dom(f_x)\colon f_x(y)\<0\}$,
\item[(ii)] $f_x(x)=0$,
\item[(iii)] $\forall\,t\in I\;\;\langle\partial f_x(x),F(t,x)\rangle^->0$ or $\forall\,t\in I\;\;\langle\partial f_x(x),F(t,x)\rangle^+<0$,
\item[(iv)] $\underset{p\in\partial f_x(x)}{\underset{y\in F(0,x)}{\max}}\langle p,y\rangle\cdot\max\limits_{z\in F(T,Cx)}f_{Cx}^\circ(Cx;-z)<0$.
\end{itemize}
Therefore, all the assumptions of Corollary \ref{wn3}. are met and the thesis of Theorem \ref{floqth}. applies.
\end{proof}
\begin{remark}
It should be noted that Theorem \ref{floqth2}. does not involve any assumptions regarding the geometry of the set $\overline{K}$, apart from the topological requirement for nonemptiness of the interior of this set. Nevertheless, if the set $\overline{K}$ is sufficiently regular, then one may alter the assumptions regarding the polar cone $T_K(x)^\circ$. For instance, if $\overline{K}$ is a proximate retract, then these assumptions can be formulated in terms of the Clarke normal cone $N_K^c(x)$.
\end{remark}

\begin{corollary}
Let $K$ be an open bounded subset of $\R{N}$ such that $\bd K$ is inavariant under the action of $\langle C\rangle$. Let $F\colon I\times\overline{K}\map\R{N}$ satisfy $(\F_1)$-$(\F_4)$. Assume that, for each $x\in\bd K$, there exists a nonzero vector $v(x)\in N_K^b(x)$ such that the following condition holds
\begin{equation}\label{con2}
\forall\,t\in I\;\;\;F(t,x)\cap\{v(x)\}^\circ=\varnothing.
\end{equation}
Moreover, assume that either assumption {\em (v)} or {\em (vi)} of Theorem \ref{floqth}. is satisfied. Then the Floquet boundary value problem \eqref{floquet} possesses a solution.
\end{corollary}
\begin{proof}
On the basis of \cite[Prop.5.1.]{deimling}, we know that there exists u-Scorza-Dragoni multimap $F_0\colon I\times\overline{K}\map\R{N}$ with compact convex values such that $F_0(t,x)\subset F(t,x)$ on $I\times\overline{K}$. Therefore for all $\n$ one may find a closed subset $I_n\subset I$ such that the Lebesgue measure $\ell(I\setminus I_n)\<\frac{1}{n}$ and the restriction of $F_0$ to $I_n\times\overline{K}$ is usc. We can assume, w.l.o.g., that the family $\{I_n\}_\n$ is increasing. \par Let $F_n\colon I\times\overline{K}\map\R{N}$ be such that $F_n(t,x):=\co F_0(P_n(t),x)$, where $P_n\colon I\map I_n$ is a metric projection onto $I_n$. It is clear that $F_n$ is a bounded usc multimap with compact convex values. Moreover, this map satisfies also hypotheses (i) and (ii) of Theorem \ref{floqth2}. Indeed, assumption \eqref{con2} entails \[\forall\,t\in I\;\forall\,y\in F_n(t,x)\;\;\;\langle v(x),y\rangle\neq 0\] and \[\forall\,y\in F_n(0,x)\;\forall\,z\in F_n(T,Cx)\;\;\;\langle v(x),y\rangle\cdot\langle v(Cx),z\rangle>0.\] \par Since the Aumann integral $\int_0^TF(t,x_0)\,dt$ is compact and the subspace $\im(id-C)$ is closed, there is $\eps>0$ such that \[B\left(\int_0^TF(t,x_0)\,dt,\eps\right)\cap\im(id-C)=\varnothing\;\text{ for }x_0\in\ker(id-C)\cap\bd K.\] One easily sees that \[\int_0^Tw(t)\,dt\in D\left(\int_0^TF(t,x_0)\,dt,\frac{2c}{n}\right)\] for every integrable selection $w$ of the multimap $F_n(\cdot,x_0)$. Thus, if we denote the map $\int_0^TF_n(t,\cdot)\,dt$ by $\w{F}_n$, then $\w{F}_n(x_0)\subset B\left(\w{F}(x_0),\eps\right)$ for $n$ large enough. Consequently, \[\int_0^TF_n(t,x_0)\,dt\cap\im(id-C)=\varnothing\;\text{ for }x_0\in\ker(id-C)\cap\bd K.\] \par Now, let us define a multivalued homotopy $H\colon[0,1]\times(\ker(id-C)\cap\overline{K})\map\ker(id-C)$ by the formula \[H(\lambdaup,x_0):=\lambdaup P_N\circ\w{F}(x_0)+(1-\lambdaup)P_N\circ\w{F}_n(x_0).\] It is clear that $H$ is an upper semicontinuous map with convex compact values, given that the projector $P_N$ is linear and $\w{F}$, $\w{F}_n$ are also convex compact valued usc multis. Observe that $\w{F}_n(x_0)\subset D(0,c+\eps)$ and as a result $H(\lambdaup,x_0)\subset D\left(0,||P_N||_{\mathcal L}(c+\eps)\right)$. The latter means that $H$ is also a compact map. Take $x_0\in\ker(id-C)\cap\bd K$. Notice that $0\not\in P_N(B(\w{F}(x_0),\eps))$. On the other hand $H(\lambdaup,x_0)=P_N(\lambdaup\w{F}(x_0)+(1-\lambdaup)\w{F}_n(x_0))\subset P_N(B(\w{F}(x_0),\eps))$. Hence, $0\not\in H([0,1]\times(\ker(id-C)\cap\bd K))$. Relying on the homotopy invariance of the topological degree for the class ${\mathcal M}$ (see \cite[11.4]{deimling}) we obtain the equality 
\begin{align*}
\deg(P_N\circ\w{F}_n,\ker(id-C)\cap K,0)&=\deg(H(0,\cdot),\ker(id-C)\cap K,0)\\&=\deg(H(1,\cdot),\ker(id-C)\cap K,0)\\&=\deg(P_N\circ\w{F},\ker(id-C)\cap K,0).
\end{align*}
We have shown, therefore, that the assumption (vi) of Theorem \ref{floqth}. is fulfilled.\par By virtue of Theorem \ref{floqth2}. we obtain a solution $x_n\in C(I,\overline{K})$, for $n$ sufficiently large, of the following Floquet boundary value problem
\[\begin{gathered}
\dot{x}(t)\in F_n(t,x(t)),\;\mbox{a.e. on }I,\\
x(T)=Cx(0).
\end{gathered}\]
Actually, $x_n$ satisfies
\[\begin{gathered}
\dot{x}_n(t)\in F(t,x_n(t)),\;\mbox{for a.a. }t\in I_n,\\
x_n(T)=Cx_n(0).
\end{gathered}\]
Keeping in mind that sequences $(x_n)_\n\subset C(I,\R{N})$ and $(\dot{x}_n)_\n\subset L^2(I,\R{N})$ are bounded one can show (compare \cite[Lem.5.1.]{deimling}) that $||x_n-x||\to 0$ and $\dot{x}_n\rightharpoonup\dot{x}$ in $L^2(I,\R{N})$ for some absolutely continuous $x\colon I\to\R{N}$. Fix $n\in\mathbb{N}$. Put $\eps_n:=||x_n-x||$. Then, w.l.o.g., $\dot{x}_m(t)\in F(t,B(x(t),\eps_n))$ a.e. on $I_n$ for every $m\geqslant n$. Define \[{\mathcal F}_{\!\eps_n}:=\left\{w\in L^2(I,\R{N})\colon w(t)\in\overline{\co}\,F\left(t,B(x(t),\eps_n)\cap\overline{K}\right)\mbox{ a.e. on }I_n\right\}.\]
Since ${\mathcal F}_{\!\eps_n}$ is convex and closed in reflexive Banach space $L^2(I,\R{N})$, it is weakly closed. Thus $\dot{x}\in{\mathcal F}_{\eps_n}$. In fact,
\[\dot{x}(t)\in\bigcap_{m=n}^\infty\overline{\co}\,F\left(t,B(x(t),\eps_m)\cap\overline{K}\right)\mbox{ a.e. on }I_n.\]
Clearly, the assumptions of Theorem \ref{convergence}. are satisfied and as a consequence $\dot{x}(t)\in F(t,x(t))$ for a.a. $t\in I_n$. Since the index $n$ was arbitrary and $\ell(\bigcup_\n I_n)=\ell(I)$, we finally obtain that $\dot{x}(t)\in F(t,x(t))$ a.e. on $I$. At the same time $x(T)=\lim\limits_{n\to\infty}x_n(T)=\lim\limits_{n\to\infty}Cx_n(0)=Cx(0)$, completing the proof.
\end{proof}
\begin{remark}
Let us pay attention to the fact that $\co T_K(x)=N_K^b(x)^\circ\subset\{v(x)\}^\circ$. Hence, \eqref{con2} entails $\forall\,t\in I\;\;F(t,x)\cap\co T_K(x)=\varnothing$ - the opposite to the Nagumo's weak tangency condition.
\end{remark}

\end{document}